\documentclass{amsart}

\usepackage[T1]{fontenc}
\usepackage{amssymb,amsthm,amsmath,mathtools,amsrefs,MnSymbol}
\usepackage{amscd,amstext,units,amsfonts,amsbsy}
\usepackage[utf8]{inputenc}
\usepackage[colorlinks=true]{hyperref}
\usepackage{color,xcolor}
\usepackage{url}
\usepackage{comment} 

\usepackage{graphicx}
\graphicspath{ {images/} }

\newcommand{\bp}{\begin{prop}}
\newcommand{\ep}{\end{prop}}
\newcommand{\bd}{\begin{definicion}}
\newcommand{\ed}{\end{definicion}}
\newcommand{\bl}{\begin{lema}}
\newcommand{\el}{\end{lema}}
\newcommand{\bh}{\begin{hecho}}
\newcommand{\eh}{\end{hecho}}
\newcommand{\bpreg}{\begin{preg}}
\newcommand{\epreg}{\end{preg}}
\newcommand{\bo}{\begin{obs}}
\newcommand{\eo}{\end{obs}}
\newcommand{\bcon}{\begin{conj}}
\newcommand{\econ}{\end{conj}}
\newcommand{\brmk}{\begin{rmk}}
\newcommand{\ermk}{\end{rmk}}
\newcommand{\bc}{\begin{corol}}
\newcommand{\ec}{\end{corol}}
\newcommand{\bconst}{\begin{const}}
\newcommand{\econst}{\end{const}}

\newcommand{\ota}{\otimes_a}
\newcommand{\opa}{\oplus_a}

\newcommand{\bitem}{\begin{itemize}}
\newcommand{\eitem}{\end{itemize}}
\newcommand{\bt}{\begin{teor}}
\newcommand{\et}{\end{teor}}
\newcommand{\be}{\begin{ejem}}
\newcommand{\ee}{\end{ejem}}
\newcommand{\bnot}{\begin{nota}}
\newcommand{\enot}{\end{nota}}

\newtheorem*{conj}{Conjecture}
\newtheorem{claim}{Claim}

\newtheorem{theorem}{Theorem}[section]

\newtheorem{corollary}[theorem]{Corollary}

\newtheorem{definition}{Definition}[section]

\newtheorem{lemma}[theorem]{Lemma}

\newtheorem{proposition}[theorem]{Proposition}

\newtheorem{fact}[theorem]{Fact}

\newtheorem{remark}[theorem]{Remark}

\numberwithin{equation}{section}

\newcommand{\ov}[1]{\overline{#1}}

\newcommand{\M}{\mathcal{M}}
\newcommand{\tp}{\operatorname{tp}}
\title{Definable groups in models of Presburger arithmetic}

\author{Alf Onshuus}
\address{Departamento de Matem\'aticas, Universidad de los Andes, Cra 1 No 18A-10, Bogot\'a 111711, Colombia}
\email{aonshuus@uniandes.edu.co}

\author{ Mariana Vicar\'ia}
\email{mariana@math.berkeley.edu}
\address{Department of Mathematics, Evans Hall, University of California at Berkeley,  Berkeley, CA, U.S.A.}

\begin{document}


\maketitle



\thanks{This research was partially subsidized by Colciencias grant number 120471250707.}     





\begin{abstract}
This paper is devoted to understand groups definable in Presburger arithmetic. We prove the following theorems:

\noindent Theorem $1$. Every group definable in a model of Presburger arithmetic is abelian- by-finite.

\noindent Theorem $2$ Every bounded abelian group definable in a model $(\mathbb{Z},+,<)$ of Presburger arithmetic is definably isomorphic to $(\mathbb{Z}, +)^{n}$ mod out by a lattice.
\end{abstract}



\section{Introduction} \label{intro}
This paper is devoted to understanding groups definable in Presburger
arithmetic. It is in the same spirit as
\cite{Pantelis} where analogous studies were made for
the theory of $(\mathbb Q, +, <)$.

In \cite{GroupsMeasuresNIP} the proof of what was known as ``Pillay's Conjecture'' was finalized. This
conjecture stated that given any definably compact group $G$ definable in an o-minimal expansion of a real closed field,
 one can find a Lie group $G_L$ as the quotient of $G$ by its largest type definable subgroup of bounded index $G^{00}$.
Moreover, not only does $G_L$ have the same dimension (as a Lie group) as the o-minimal dimension of $G$, but the pure group
theory of $G_L$ and $G$ are the same, meaning that phenomena such as abelianity, definable solubility, etc. of $G$ are already captured by
$G_L$. This result was then proved in \cite{Pantelis} for groups definable in $(\mathbb Q, +, <)$.

One would like to extend these results (understanding definable
groups in terms of more standard geometric objects) to groups
definable in other geometric contexts such as the theory of the
$p$-adics $Th(\mathbb Q_p, +, \cdot)$, or Presburger artihmetic
(that is, the theory $Th(\mathbb Z, +, <)$). Notice that
Presburger arithmetic is a reduct of $Th(\mathbb Q_p, +, \cdot)$
(it is the value group) so understanding groups definable in
Presburger would be a first step in the understanding of groups
definable in $p$-adically closed fields. In this paper we find
analogous results to those proved in \cite{Pantelis} for
Presburger arithmetic.

We will prove the following theorems, which are Theorems
\ref{ab-by-finite} and \ref{THEOREM}.

\begin{theorem} Every infinite group definable in a model of Presburger arithmetic is
abelian-by-finite.
\end{theorem}

\begin{theorem} Every infinite and bounded group $G$ definable in a model $(Z,+,<)$ of Presburger arithmetic of dimension $n$ is
definably isomorphic to $(Z, +)^n/ \Lambda$ where $\Lambda$ is a local lattice of $(Z,+)^{n}$.

More precisely, there is an $n$-dimensional box $B\in Z$ and a
local $B$-lattice $\Lambda$ (see Definition \ref{locallattice})
such that $G$ is isomorphic to the ($\wedge$-definable) subgroup
of $Z^n$ generated by $B$ modulo $\Lambda$.
\end{theorem}

We will always work in Presburger arithmetic, and set
\[T=Th(
Z,+,<,-,\equiv_{n},0,1)\] so that we have quantifier elimination.

This paper combines results from an unpublished preprint by the
first author and from the second author's Master's thesis at
Universidad de los Andes.

\section{Preliminaries}
\label{Preliminares}
We are interested in understanding the definable groups  in
Presburger arithmethic. For this we are
going to use some facts that are already known for this particular
theory, such as quantifier elimination and the Cell Decomposition
Theorem. For sake of completeness we present all the statements
that we are going to use throughout this paper, and their respective
references.

We will always use $\mathcal{M}$ to indicate a saturated model
of Presburger arithmetic, and $\mathcal{M}_{0}$ will be a small
elementary submodel of $\mathcal{M}$. Also, $(G,\cdot, e)$
 denotes a definable group over  $\mathcal{M}_{0}$
and usually $B$ denotes a small set of parameters from
$\mathcal{M}$. Unless otherwise specified, given a definable set
$X \subseteq \mathcal{M}^{n}$ and an element $x \in X$ we denote
by $x_i$ the $i$-th coordinate of $x$ so that
$x=(x_{1},\dots,x_{n})$.

\subsection{Quantifier Elimination and Cell Decomposition}\label{sub1.1}
We begin with some facts about Presburger arithmetic. The
following is Corollary $3.1.21$ in \cite{Marker}.

\begin{fact}
Let $\mathcal{L}_{Pres}=\{+,-,<,\{\equiv_{n}\}_{n \in \mathbb{N}}
,0,1\}$, where Presburger
arithmetic has quantifier elimination in $\mathcal{L}_{Pres}$. An explicit axiomatization can be found in \cite{Marker}.
\end{fact}

\begin{corollary} \label{terminos} Let $x$ be a single variable
and $a=(a_{1},\dots,a_{k})$ be a tuple of elements in
$\mathcal{M}$. Then any $\mathcal{L}_{Pres}$- term $\tau(x,a)$ is of the form:
\begin{align*}
sx+ \sum_{i=1}^{n} k_{i} a_{i}+l, \ \text{where $s ,k_{i}, l$ are  integers}.
\end{align*}
\end{corollary}

\begin{proof}
It follows by induction of the length of $\tau(x,a)$.
\end{proof}

\medskip

The following is a definition, which corresponds to ``linear and $B$-definable''
in \cite{cell}.

\begin{definition} Let $f: X \subseteq \mathcal{M}^{m} \rightarrow \mathcal{M}$
be a function. We say that $f$ is \emph{$B$-linear}, if it can be
written in the form
\begin{align*}
\displaystyle{f(x)= \sum_{i=1}^{m} s_{i}\left( \frac{x_{i}-c_{i}}{k_{i}}\right)} +\gamma,
\end{align*}
where $\gamma \in dcl(B)$, and for each $1\leq i \leq m$, $0\leq
c_{i} < k_{i}$ and $s_i$ are integers, such that
$x_{i}\equiv_{k_{i}} c_{i}$ and $x_{i}$ is the $i$-th coordinate of $x$. \end{definition}

\begin{proposition} \label{atomic} Let $\sigma(x,\overline{a})$ be a consistent
 atomic $\mathcal{L}_{Pres}$-formula, with parameters. Then
$\sigma(x,\overline{a})$  is equivalent to a formula of one of the
following forms: \begin{enumerate} \item $x=b$, where $b=t(\ov{a})$ and $t$
is an $\emptyset$-linear function. \item $x\leq b$, where
$b=t(\ov{a})$ and $t$ is an $\emptyset$-linear function. \item $x
\geq b$, where $b=t(\ov{a})$ and $t$ is an $\emptyset$-linear
function. \item $x \equiv_{N} c$, where $0\leq c < N$ are
integers. \end{enumerate} \end{proposition}

\begin{proof} It is left as an exercise to the reader.
 \end{proof}

We will now state the Cell Decomposition Theorem due to R.
Cluckers (see \cite{cell}).

\begin{definition} [$B$-definable cells] Let $B$ a set of parameters. We will define  \emph{$B$-definable
cells} inductively as follows:

\begin{enumerate} \item A \emph{$B$-definable $0$-cell} is a point $p \in
dcl(B)$. A \emph{$B$-definable $1$-cell} is an infinite set of
the form $\{ \alpha \square_{1} x \ \square_{2} \beta \ | \ x
\equiv_{N} c \}$, where $\alpha, \beta \in dcl(B)$, $0 \leq c < N$
are positive integers and $\square_{j}$ stands for either $\leq$
or no condition.

\item Assume that we have defined an $(i_{1},\dots, i_{n})$
$B$-definable cell $C$, where $i_{j} \in \{0, 1\}$ for all $j \leq
n$. Then

\bitem \item A \emph{$B$-definable $(i_{1},\dots,i_{n},0)$-cell}
$D$ is a set of the form\\
\[
\{ (x,t) \in \mathcal{M}^{n+1} \ | \ \ov{x} \in
C \wedge t=\alpha(x,b)\}
\]

where $\alpha$ is an $B$-definable linear function. In this case we denote $D$ as a
graph $\Gamma(C,\alpha(x,b))$ of the function $\alpha$
over the domain $C$.

\item A \emph{$B$-definable $(i_{1},\dots, i_{n},1)$-cell} is a
set of the form
\[
D= \left\{ (x, t)\in \mathcal{M}^{n+1} \ | \ x \in C
\wedge \left( (\alpha(x,\ov{b}) \square_{1} t \square_{2}
\beta(x,b)) \ | \ t \equiv_{N} k \right)\right\}
\]
where $\alpha, \beta$ are $B$-definable linear functions defined on $C$, $0\leq k < N$ are integers, the symbol $\square_{i}$
represents either $\leq$ or no condition for $i=1,2$ and the size of the  fibers $F_{x}=\{ t \in \mathcal{M} \ | \ (x,t) \in D  \}$
cannot be uniformly bounded over $C$,
meaning that there is no $N\in\mathbb{N}$ such that for all $x\in C$, $|F_{x}|\leq N$.

In this case we denote $D$ as $\left((\alpha(x,b)),
\beta(x,b)\right)_{C}$. \eitem \end{enumerate}
\end{definition}

\begin{definition} [Partition into cells] Let $X$ be a definable set. A
\emph{partition $\mathcal{P}$ of $X$} is a finite set $\{
C_{1},\dots, C_{n}\}$ of pairwise disjoint $B$-definable cells
such that $\displaystyle{X= \bigcup_{i=1}^{n} C_{i}}$.
\end{definition}

\begin{definition} [Piecewise $B$-linear function] Let $X$ be a definable set and
$f: X \rightarrow \mathcal{M}$ a definable function.
We say that $f$ is a piecewise $B$-linear function, if there is a partition $\mathcal{P}$ of $X$ into
$B$-definable cells such that for each $C \in \mathcal{P}$ $f{\upharpoonright}_{C}: C \rightarrow \mathcal{M}$ is a $B$-linear function.
\end{definition}

The following theorem is the Cell Decomposition Theorem for
Presburger arithmetic, Theorem $1$ in \cite{cell}.

\begin{fact} \label{celldecomposition} (Cell Decomposition Theorem).
\begin{enumerate} \item Let $X$ be an $B$-definable set. Then there is a
partition $\mathcal{P}$ of $X$ in $B$-definable cells. \item Let
$X \subseteq \mathcal{M}^{n}$ and $f: X \rightarrow \mathcal{M}$
be definable, then $f$ is a piecewise linear function. Moreover,
if $X$ and $f$ are $B$-definable, then we can take the cells also
$B$-definable. \end{enumerate} \end{fact}

The following follows immediately from the above fact and the definition of piecewise $B$-linear.

\begin{corollary} \label{clausuradefinible} For any set $B$ and any tuple $b$ we
have $b \in dcl \left( B\right)$ if and only if there is an
$\emptyset$-definable linear function $\alpha$ and a tuple $a
\subseteq B$ such that $\alpha(a)=b$.
\end{corollary}

\begin{remark} By the previous corollary, every linear function definable
over $B$, \[\displaystyle{\alpha(x)= \sum_{i=1}^{n} s_{i} \left(
\frac{x_{i}-c_{i}}{n_{i}}\right)+\gamma'}\] can be seen as
$\alpha(x,b)= f(x)+\gamma(b)$, where
$\displaystyle{f(x)}=\sum_{i=1}^{n} s_{i} \left(
\frac{x_{i}-c_{i}}{n_{i}}\right)$, $\gamma$ is an
$\emptyset$-definable linear function, $b$ is a tuple of elements
of $B$ and $\gamma'=\gamma(b)$.
 \end{remark}

\subsection{Dimension for definable sets in Presburger}\label{sub1.2}
For any $(i_{1},\dots,i_{n})$-cell $C$ we define
$\displaystyle{dim_{cell}(C)= \sum_{j=1}^{n}i_{j}}$, and by cell
decomposition we can extend this notion of dimension to any
definable set as
\[dim_{cell}(X)= max \{ \dim_{cell}(C) \ | \ C \
\text{ is a cell and } C\subseteq X\}.\]

Belegradek, Peterzil and Wagner proved that Presburger arithmetic
has the Exchange property (see \cite{quasiomin}), so the dimension
of a definable set $X$ is well defined. We introduce the exact
definition.

\begin{definition}  \begin{enumerate} \item Let $a=(a_1,\ldots,a_n)\in \mathcal{M}^n$ and $B$ a
set of parameters. The \emph{dimension of $a$ over $B$} is
the size of a maximal independent subset of $\{a_1,\ldots,a_n\}$
with respect to definable closure. Namely, $\dim(a)=k$ if
there are $a_{i_1},\ldots,a_{i_k}\in \{a_1,\ldots,a_n\}$ such
that: \bitem \item $a_{i_1}\not\in dcl(B)$, and
$a_{i_{j+1}}\not\in dcl(a_{i_1},\ldots,a_{i_j},B)$ for
$j=1,2,\ldots,k$. \item $a_j\in dcl(a_{i_1},\ldots,a_{i_k},B)$ for
all $j=1,\ldots,n$. \eitem

\item Let $X\subseteq \mathcal{M}^n$ be a $B$-definable set. We
define the \emph{dimension of $X$} by:\[\dim_{def}(X)=\max
\{\dim(a/B):a\in X\}.\]\\
Moreover, we say that $a \in X$ is a dim-generic point if $\dim(a/B)=\dim(X)$.
\item Given two points $a,b \in \mathcal{M}^{n}$ we say that $a$ is independent from $b$ over $A$ if $dim(a/Ab)=dim(a/A)$.
In general, if $X \subseteq \mathcal{M}^{n}$ is a definable, given two points $a,b \in X$ we say that they are independent
if they are independent over the parameters defining $X$.
 \end{enumerate}
 \end{definition}

The following is Corollary 1 in \cite{cell}.
\begin{fact}
Let $X$ be any definable set in Presburger arithmetic. Then
$dim_{cell}(X)=dim_{def}(X)$. Thus, we will indicate the dimension
of $X$ as $dim(X)$. \end{fact}

In \cite{cell} R. Cluckers proved that Presburger has elimination
of imaginaries, so that by (for example) Theorem 4.12 in
\cite{EaOn} we have that Presburger is rosy with
$\text{U}^{\text{\th}}$-rank equal to the dimension defined above.
In particular we have the following properties.

\begin{itemize}
\item For any two definable sets $X,Y$, $\dim(X \cup
Y)=\max\{\dim(X),\dim(Y) \}$.

\item If $f:X \rightarrow Y$ is a definable surjective function,
such that $dim(f^{-1}(a))=n$ for some $n\in \mathbb N$, then
$\dim(Y)+n= \dim(X)$.

\item If $f:X \rightarrow Y$ is a definable surjective function,
such that $dim(f^{-1}(a))\geq n$ for some $n\in \mathbb N$, then
$\dim(Y)+n\leq \dim(X)$.
\end{itemize}

We have the following.

\begin{remark} \label{dimfiniteindex} Let $G$ be a definable group in
$\mathcal{M}$. Let $H$ be a definable subgroup of $G$. Then
$dim(H)=dim(G)$ if and only if $H$ has finite index in $G$.
\end{remark}

Given a definable set $X$, we will say that a tuple $a\in C$ is
\emph{dim-generic} if $\dim(a)=\dim(C)$. Two tuples $a$ and $b$
are \emph{independent} if $\dim(ab)=\dim(a)+\dim(b)$.

\begin{remark}\label{matriznot}
For each $i \leq k$, let $f_{i}: X_{i} \subseteq \mathcal{M}^{m}
\rightarrow \mathcal{M}$ be a linear function definable over a
small set of parameters $B$. Thus, for each $i\leq k$, there are
integers $s^{i}_{j}$, $0\leq c^{i}_{j} < n^{i}_{j}$ and an
element $\gamma_{i} \in dcl(B)$ such that:
\begin{align*}
f_{i}(x)= \sum_{j=1}^{m} s^{i}_{j}\left( \frac{x_{j}-c^{i}_{j}}{n^{i}_{j}}\right)+\gamma_{i}.
\end{align*}
 Let  $\displaystyle{X=\bigcap_{i=1}^{k}X_{i}}$ and define a function $g: X  \rightarrow \mathcal{M}^{k}$,
 such that  $g(x)=(f_{1}(x),\dots,f_{k}(x))$. \\

 A matrix representation of $g$ is $A \in \mathbb{Q}_{k \times m}$
  a tuple $c \in \mathbb{Z}^{mk}$ and a vector $\gamma \in dcl(B)^{k}$, where:
\[
A= \left( \frac{s_{j}^{i}}{n_{j}^{i}}\right)_{1 \leq i \leq k, 1 \leq j \leq m}
\]
and $\gamma$ is a vector
$
\gamma=[\gamma_{1},\gamma_{2}, \dots, \gamma_{k}]^{T}$ and we define $c=(c_{1}^{1},\dots,c_{m}^{1},c_{1}^{2},\dots,c_{m}^{2},\dots, c_{1}^{k},\dots,c_{m}^{k})$.

So $g(x)$ is achieved by multiplying each row of the matrix $A$ by the translate $x-{c^i}$ of $x$ and then adding $\gamma$.

We will abuse notation and use
$A_{c}(x)+\gamma$ to refer to $g(x)$.
\end{remark}

\section{Definable groups in Presburger arithmetic are abelian-by-finite}\label{sec2}

We will now prove that every definable group in
Presburger arithmetic is abelian-by-finite. We begin with
some notation.

\subsection{Open cells and boxes}\label{sub1.4}
\begin{definition} Let $C \subseteq \mathcal{M}^{n}$ be an
$\mathcal{M}_{0}$-definable cell. We say that $C$ is an \emph{open
cell} if $dim(C)=n$. \end{definition}

\begin{definition} [Box around a point $a$]
Let $a=(a_{1},\dots,a_{n})$ be a point in $\mathcal{M}^{n}$, we
define a box $B$ around $a$ as a product of $n$ $1$-cells $B=B_{1}
\times \dots \times B_{n}$, where for each $i \leq n$,  $a_{i} \in
B_{i}= \{ \alpha_{i} \square_{1} x \square_{2} \beta_{i} \ | \
x\equiv_{N_{i}}c_{i}\}$ and both sets $[\alpha_{i}, a_{i}]$,
$[a_{i}, \beta_{i}]$ are infinite.

 \end{definition}

\begin{lemma} \label{cajaendos} Let $C$
be a $(1,1)$-cell definable over $\mathcal{M}_{0}$. Let $(a,b)$ be
a point of $C$ of dimension 2, then there is a definable box $B$,
such that $B \subseteq C$ and $B$ is a box around $(a,b)$.
\end{lemma}

\begin{proof} Assume that $C=\{
(x,t) \ | \ x \in D \wedge \alpha(x) \leq t \leq \beta(x) \wedge t
\equiv_{m}d\}$, where $D$ is a $1$-cell of the form $D= \{ \gamma
\leq x \leq \delta \ |  \ x \equiv_{n} c\}$. Without loss of
generality, we may assume that $\alpha$ and $\beta$ are not
constant functions and neither $\pm \infty$. Then we have the following cases:

\begin{enumerate} \item $\beta$ and $\alpha$
are both increasing,

\item $\beta$ is increasing and $\alpha$ is decreasing,

\item $\beta$ is
decreasing and $\alpha$ is increasing,

\item $\beta$ and $\alpha$ are both
decreasing.
\end{enumerate}

We will only show how to solve the first case, because the others follow
in a similar way. So assume that $\beta$ and $\alpha$ are both
increasing. We consider the following types.
\begin{align*}
\Sigma_{1}(x)&= \{ x \in D \} \cup \{ x < a\} \cup \{ \beta(x)> b\} \cup \{(x \neq d) : d \in dcl(Ba) \cap D \wedge d < a \} \\
&\cup \left\{ \exists y_{1},\dots,y_{n} \left(\left( \bigwedge_{i=1}^{n} b< y_{i}< \beta(x) \right) \wedge \left( \bigwedge_{i\neq j} y_{i} \neq y_{j} \right)\right):n<\omega\right\}, \\
\Sigma_{2}(x)&= \{ x \in D \} \cup \{ x > a\} \cup \{ \alpha(x)< b\} \cup \{(x \neq d) : d \in dcl(Ba) \cap D \wedge d > a \} \\
&\cup \left\{ \exists y_{1},\dots,y_{n} \left(\left(
\bigwedge_{i=1}^{n} \alpha(x)< y_{i} < b \right) \wedge \left(
\bigwedge_{i\neq j} y_{i} \neq y_{j}
\right)\right):n\in\omega\right\}.
\end{align*}

By compactness both $\Sigma_{1}(x)$ and $\Sigma_{2}(x)$ are
consistent so by saturation of $\mathcal{M}$ there are elements
$p,q$ such that $\mathcal{M} \vDash \Sigma_{1}(p)$ and
$\mathcal{M} \vDash \Sigma_{2}(q)$. Define the following
$1$-cells:
\begin{align*}
I_{1}&=\{ p \leq x \leq q \ | \ x \equiv_{n} c \} \ \text{and} \ I_{2}=\{ \alpha(q) \leq t \leq \beta(p) \ | \ t \equiv_{m} d \} .
\end{align*}

So $(a,b) \in I_{1}\times I_{2}$, and by construction $I_{1}
\times I_{2}$ is a box around $(a,b)$. Since $I_{1} \subseteq D$,
and for every element $x \in I_{1}$, $p \leq x \leq q$, we have
that $\beta(p) \leq \beta(x) $ and $\alpha(x) \leq \alpha(q)$, so
$I_{1} \times I_{2} \subseteq C$. \end{proof}

\begin{lemma} \label{caja} Let $C$ be an open $n$-cell  definable over
$\mathcal{M}_{0}$ and $a$ be a dim-generic element of $C$. Then
we can find an $n$-box $B$ such that $B \subseteq C$ and $B$ is a box around
$a$.
\end{lemma}

\begin{proof} We proceed by induction on $n$. Let $C= \{ \alpha
\square_{1} x \square_{2} \beta \ | \ x\equiv_{N} c\}$, where
$\alpha, \beta \in \mathcal{M}_{0}$. Since $a$ is dim-generic, $[\alpha,a]$ and $[a,\beta]$
are both infinite,  so $C$ is already a box around $a$.

Let $C$ be an open $(n+1)$-cell, say of the form $\{ (x,t) \
| \ x \in D \wedge \alpha(x) \leq t \leq \beta(x)
\wedge t \equiv_{m} c\}$, where $D$ is an $n$-open cell and
$\alpha$ and $\beta$ are linear functions. Let
$(a_{1},\dots,a_{n+1})$ be a dim-generic point of $C$. By
induction there is a box $S$ around $(a_{1},\dots,a_{n})$
such that $S \subseteq D$. Assume that $S=I_{1} \times \dots
\times I_{n}$, and that for each $ i \leq n$, $I_{i}=\{
\delta^{i}_{1} \leq x \leq \delta^{i}_{2} \ | \  x \equiv_{k_{i}}
l_{i}\}$. Now define the following lines inside $S$
\begin{align*}
L_{1}&=I_{1} \times \{ a_{2}\} \times \cdots \times\{a_{n}\},\\
&\vdots\\
L_{n} &= \{a_{1}\} \times \cdots \times \{a_{n-1}\} \times I_{n}.
\end{align*}
Observe that $a_{i} \in L_{i}$ for every $ i \leq n$, and the
restriction of $\alpha(x)$ and $\beta(x)$ to $L_{i}$ are
lines. Let $\alpha_{i}$ and $\beta_{i}$ be the restrictions of
$\alpha(x)$ and $\beta(x)$ to $L_{i}$.

In order to simplify the notation, we will assume that all of
$\alpha_i$ and $\beta_i$ are increasing (the general result will
follow similarly).

Since $dim(a_{n+1}/ \mathcal{M}_{0},a_{1},\dots,a_{n})=1$,
following the argument of Lemma \ref{cajaendos} for each $L_{i}$
and $\alpha_{i}(x)$ and $\beta_{i}(x)$, we can find
$\delta_{1}^{i} \leq p_{i} < a_{i} < q_{i} \leq \delta^{i}_{2}$
such that the following hold:

\bitem \item The intervals $[p_{i},a_{i}]$ and $[a_{i},q_{i}]$ are
both infinite.

\item The intervals $[a_{n+1}, \beta_{i}(p_{i})]$ and
$[\alpha_{i}(q_{i}), a_{n+1}]$ are infinite.

\item The set $\{ p_{i} \leq x \leq q_{i} \ | \ x
\equiv_{k_{i}}l_{i}\} \times \{ \alpha_{i}(q_{i}) \leq t \leq
\beta_{i}(p_{i}) \ | \ t \equiv_{m} c\}$ is a $2$-box and it is a
subset of the $2$-open cell $
\{ (x,t) \ | \ x \in I_{i} \wedge \alpha_{i}(x) \leq t \leq
\beta_{i}(x) \wedge t\equiv_{m}c\}.$

Define $J_{i}=\{ p_{i} \leq x \leq q_{i} \ | \ x \equiv_{k_{i}}
l_{i}\}$ (notice that $a_{i} \in J_{i} \subseteq I_{i}$), $r= \max
\{ \alpha_{i}(p_{i}) \ | \  i \leq n\}$ and $R=\min\{\beta_{i}(p_{i}) \ |  i \leq n \}$.\\
Now, let $B= J_{1}\times \dots \times J_{n} \times \{ t \in \mathcal{M} \ | \ r \leq t \leq R \ | \ t \equiv_{m} c\},$ so $B$ is a box around $(a_{1},\dots,a_{n+1})$ and $B
\subseteq C$, as required.  \qedhere \eitem \end{proof}

\begin{remark}\label{intersectionboxes} Let $B_1,B_2$ be two boxes around a
point $a$, then $B_1 \cap B_2$ is a box around $a$.
\end{remark}

If $\bar{\imath}$ be a sequence of 0's and 1's, let $C$ be an $\bar{\imath}$-cell and let $\pi_{\bar{\imath}}$ be the projection of $C$ into the coordinates
where $\bar{\imath}$ has 1's. By definition of $\bar{\imath}$-cell, $\pi$ is a bijection of $C$ into $\pi_{\bar{\imath}}(C)$ and $\pi_{\bar{\imath}}(C)$ is an open cell.

\begin{definition}
let $C$ be an $\bar{\imath}$-cell, and let $n$ be the dimension of
the image of $\pi_{\bar{\imath}}$ (the number of ones in
$\bar{\imath}$). We define \emph{$C$-boxes} to be the preimages
under $\pi_{\bar{\imath}}\upharpoonright_{C}$ of $n$-boxes in
$Z^n$ contained in $\pi_{\bar{\imath}}(C)$.
\end{definition}

Notice that $B_c$ is a $C$-box if and only if
$B_c=\pi_{\bar{\imath}}^{-1}(B)\cap C$ for some open box $B\subset Z^n$
such that $dim(\pi_{\bar{\imath}}^{-1}(B)\cap C)=dim(C)$.

\bigskip

The following hold either from the definition, or from the
previous results using the fact that the restriction of
$\pi_{\bar{\imath}}$ to $C$ is a bijection into an open $n$-cell.

\begin{corollary} \label{cajabi} \label{intersectionCboxes} Let $C$ be a $\bar{\imath}$-cell definable over
$\mathcal{M}_{0}$ and let $a$ be a dim-generic element of $C$. Then
the following hold:

\begin{itemize}
\item There is a $C$-box $B$ such that $B \subseteq C$ and $B$ is a box around $a$.

\item Let $B_1,B_2$ be two $C$-boxes around $a$. Then $B_1
\cap B_2$ is a $C$-box around $a$.

\item If $D\supset C$ is a $\bar{\imath}$-cell, then any $C$-box is a
$D$-box.
\end{itemize}
\end{corollary}

\subsection{Every group operation is locally linear}\label{sub1.5}

Throughout this subsection, we will fix a definable group $G$ and
a cell decomposition of $G$. Unless otherwise specified, all cells
we refer to will be cells in this cell decomposition. Assume that
$G\subset Z^n$ and that it has dimension $d$.

The main purpose of this subsection is to show that for every
definable group $(G,\cdot, e)$ there is a dim-generic cell $C$ in
the cell decomposition of $G$ and a $C$-box $B_{a} \subseteq G$,
such that for every $x,y \in B_{a}$, we have that:
\begin{align*}
x \cdots a^{-1} \cdots y= x -a + y.
\end{align*}
For notation purposes, we will write $xy$ instead of $x\cdot y$
whenever $x$ and $y$ are elements in $G$.

Our proof is based on the work of Eleftheriou and Starchenko in
\cite{Pantelis}, where they prove the same result for definable
groups in vector spaces over division rings. We adapt many of
their methods to the context of Presburger arithmetic.

For the rest of the paper we will introduce notation for ``addition centered in $a$'' and
``multiplication centered in $a$''. So let $x\otimes_a y:=xa^{-1}y$ and $x\oplus_a y:=x-a+y$.

\bigskip

\begin{lemma}\label{celdassuma} Let $a, b$ be two dim-generic and independent points in $G$.
Let $C_1$ and $C_2$ be
${\bar{\imath}}_1$ and ${\bar{\imath}}_2$-cells containing $a$ and
$b$ respectively and such that $(a, b)$ is dim-generic
over the parameters defining $G, C_1$ and $C_2$.

Then there are $C_1$ and $C_2$-boxes
$B_{a}$  and $B_{b}$ around $a$ and
$b$ respectively, and  $\mathcal{M}_0$-definable linear
functions $f_1,\ldots,f_n$  defined on $B_{a}\times
B_{b}$ such that for every $(x,y) \in
B_{a}\times B_{b}$, we have that
\[xy=
\left(f_{1}(x,y),\ldots,f_n(x,y)\right).\]
Moreover, there are matrices $M,N \in \mathbb{Q}_{n\times n}$,
tuples $c,d \in \mathcal{M}^{n^{2}}$ and a vector
$\gamma \in \mathcal{M}_{0}^{n}$, such that for all $ (x,y) \in B_{a} \times B_{b}$
\[
xy= M_{c} x+ N_{d} y+ \gamma.
\]\end{lemma}

\begin{proof} The multiplication in $G$ is a definable function, so
for every $i \leq n$ the projection $\pi_{i} \circ \cdot$ on the
$i$-th coordinate is $\mathcal{M}_{0}$-definable.

By cell decomposition, we can find a partition $\mathcal{P}$ of
cells of $G\times G$ such that for every $i \leq n$ and $D' \in
\mathcal{P}$, $\pi_{i} \circ \cdot{\upharpoonright}_{D'}$ is an
$\mathcal{M}_{0}$-definable linear function.

Let $D \in \mathcal{P}$ be the cell containing
$(a,b)$ (so by definition it must be a
$({\bar{\imath}}_1\widehat{\ \ }{\bar{\imath}}_2)$-cell) and by dim-genericity of
$(a,b)$ there is some $({\bar{\imath}}_1\widehat{\ \
}{\bar{\imath}}_2)$-cell $D_0$ contained in both $C_1\times C_2$ and in $D$, and
containing $(a,b)$. By Corollary \ref{cajabi}, there is
a $D_0$-box $B$ around the point $(a,b)$ completely
contained in $D_0$. Let $k_{1}$ and $k_{2}$ be the number of $1$'s in ${\bar{\imath}}_1$ and ${\bar{\imath}}_2$ respectively.
Assume that $B= \pi_{({\bar{\imath}}_1\widehat{\ \
}{\bar{\imath}}_2)}^{-1}( J_{1}\times \dots \times J_{k_{1}}
\times J_{k_{1}+1} \times \dots \times J_{k_{1}+k_{2}})$, where each $J_i$ is an interval.
Define $B_{a}=\pi_{({\bar{\imath}}_1\widehat{\ \
}{\bar{\imath}}_2)}^{-1}(J_{1} \times
\dots \times J_{k_{1}})$ and $B_{b}=\pi_{({\bar{\imath}}_1\widehat{\ \
}{\bar{\imath}}_2)}^{-1}( J_{k_{1}+1} \times \dots \times
J_{k_{1}+k_{2}})$. These are $C_1$ and $C_2$-boxes around $a$ and
$b$ respectively, and they satisfy the desired condition.

The ``moreover'' part follows by taking the matrix representation
of $g= (f_{1},\dots,f_{n})$ (taking for $M$ and $c$ the part
of the function involving $x$, and for $N$ and $d$ the
part involving $y$). \end{proof}

\begin{lemma} \label{celdassumaadentro} Let $a$, $b
\in G$ be two dim-generic and independent points belonging to
${\bar{\imath}}_1$ and ${\bar{\imath}}_2$-cells $C_1$ and $C_2$. Then there are
$C_1$ and $C_2$-boxes $B_{a}$ and $ B_{b}$
around $a$ and $b$, respectively, completely contained
in $G$, and matrices $M,N \in \mathbb{Q}_{n \times n}$, tuples
$c,d \in \mathbb{Z}^{n^2}$ and vectors $\gamma_{1},
\gamma_{2} \in dcl(\mathcal{M}_{0},a^{-1} b)^{n}$ such
that the following hold:

\bitem \item For all $x \in B_{a}$, we have
$x  a^{-1}  b=
M_{c} x+\gamma_{1} \in B_{b}$.

\item For every $x \in B_{a}$, we have $
a^{-1}  b  x
=N_{d}x+\gamma_{2} \in B_{b}$.

\eitem
\end{lemma}

\begin{proof} Since $a$ and $b$ are independent and
dim-generics of $G$, then $a$ and $a^{-1}
 b$ are independent and dim-generics of $G$ as
well. Therefore, if $a^{-1}  b$ lives
in a ${\bar{\imath}}_3$-cell $C_3$,  by Lemma \ref{celdassuma} there are
$C_1$ and $C_3$-boxes $B^{0}_{a}$ and
$B_{a}^{-1}  b$, matrices $M, N \in
\mathbb{Q}_{n\times n}$, tuples $c, d \in
\mathbb{Z}^{n^{2}}$ and an element $\gamma \in
\mathcal{M}_{0}^{n}$, such that for all $x \in
B^{0}_{a}$ and $y \in B_{a}^{-1}
 b$ we have
\begin{align*}
x  y= M_{c} x+N_{d} y+\gamma.
\end{align*}

This implies that for any element $x \in B^{0}_{a}$, $x  a^{-1}
 b= M_{c}
x+N_{d}(a^{-1}  b)+\gamma$. Defining
$\gamma_{1}= N_{d}(a^{-1}
b)+\gamma$, for each element $x \in
B^{0}_{a}$ we get
\[
x  a^{-1}  b=M_{c} x+\gamma_{1}.
\]

Consider now $D_{b}=\{ x a^{-1}  b \ | \ x \in B^{0}_{a}\}$, which
is an $\mathcal{M}_{0} \cup \{ a^{-1}  b\}$-definable set. Let
$B'_{b}$ be a $C_2$-box such that $b \in B'_{b} \subseteq D_{b}$.
Similarly, define $D_{a}= \{ x \in B'_{a} \ | \ x a^{-1}  b \in
B'_{b}\}$, which is an $\mathcal{M}_{0} \cup \{ a^{-1}
b\}$-definable set and $a \in D_{a}$. Again we can find a
$C_1$-box $B_{a}'$ around $a$ such that $B_{a}' \subseteq D_{a}$.
These boxes satisfy the first condition of the lemma. Likewise we
can obtain $C_1$ and $C_2$-boxes $B^{''}_{a}, B^{''}_{b}$ which
satisfy the second condition. Take $ B_{a}= B'_{a} \cap
B^{''}_{a}$ and $B_{b}= B'_{b} \cap B^{''}_{b}$.
\end{proof}

\begin{lemma} \label{vecindadcool} Let $a$ be a dim-generic
element in a ${\bar{\imath}}$-cell $C$. Then there is a $C$-box
$B_{a}$ around $a$, matrices $M,N,P,Q \in
\mathbb{Q}_{n\times n}$, tuples $c,d,e,f \in
\mathbb{Z}^{n^{2}}$ and vectors $\gamma_{1},\gamma_{2},
\beta \in \mathcal{M}^{n}$, such that
\[ x  a^{-1}  y= P_{e}( M_{c} x+\gamma_{1})+ Q_{f} (N_{d} y+\gamma_{2})+ \beta.
\]for all $x,y \in B_{a}$.\end{lemma}

\begin{proof} Let $a$ be a dim-generic element of $G$ and take
$a_{1}$ a dim-generic element of $G$ independent of
$a$. Consider $a_{2}=a
a_{1}^{-1}$, which is also dim-generic and independent
of $a$. Assume that $a_1$ lives in a
${\bar{\imath}}_1$-cell $C_1$ (in the cell decomposition of $G$) and that
$a_2$ lives in a ${\bar{\imath}}_2$-cell $C_2$.

Applying  Lemma \ref{celdassumaadentro} to $a$ and $a_{1}$ we
obtain $C$ and $C_1$-boxes $B'_{a}$ and $B'_{a_{1}}$, a matrix $M
\in \mathbb{Q}_{n \times n}$, a tuple $c \in \mathbb{Z}^{n^{2}}$
and a vector $\gamma_{1} \in \mathcal{M}^{n}$ such that for all $x
\in B'_{a}$ we have
\begin{align*}
x  a^{-1}  a_{1}&= M_{c}x+\gamma_{1} \in B'_{a_{1}}.
\end{align*}

Similarly, since $a$ and $a_{2}$ are also
independent and dim-generics, there are $C$ and $C_2$-boxes
$B^{''}_{a}$ and $B^{''}_{a_{2}}$, a matrix
$N \in \mathbb{Q}_{n\times n}$, a tuple $d \in
\mathbb{Z}^{n^{2}}$ and a vector $\gamma_{2} \in
\mathcal{M}^{n}$ such that for all $x \in
B^{''}_{a}$
\[
a^{-1}  a_{2}  x =
N_{d} {y}+\gamma_{2} \in B^{''}_{a_{2}}.
\]

Additionally, applying Lemma \ref{celdassuma} to
$a_{1}$ and $a_{2}$, which are also
independent and dim-generics,  there are $C_1$ and $C_2$-boxes
$B{'''}_{a_{1}}$ and $B^{'''}_{a_{2}}$,
matrices $P,Q \in \mathbb{Q}_{n\times n}$, tuples $e,f
\in \mathbb{Z}^{n^{2}}$ and a vector $\beta \in \mathcal{M}^{n}$
such that for any element  $x \in
B^{'''}_{a_{1}}$ and $y \in
B^{'''}_{a_{2}}$ we have that $x
y= P_{e}x+ Q_{f}y+ \beta$.

Now we can consider the ($C$, $C_1$ and $C_2$)-boxes
$B_a=B'_{a}\cap B^{''}_{a}$, $B_{a_{1}}=B^{'''}_{a_{1}} \cap
B'_{a_{1}}$ and $B_{a_{2}}=B^{'''}_{a_{2}} \cap B^{''}_{a_{2}}$.
Without loss of generality, we may assume that $x a^{-1} a_{1} \in
B_{a_{1}}$ and $a^{-1}
 a_{2}  x \in B_{a_{2}}$ for any element
$x \in B_{a}$ (we can always find a smaller box around
$a$ that satisfies these conditions following the same
argument of the last part of Lemma \ref{celdassumaadentro}).

Thus, for any $x,y \in B_{a}= B_{a}'\cap B_{a}^{''}$, we have
\begin{align*}
x   a^{-1}  y =
x   a^{-1}  a_{1}
a_{1}^{-1}  y & = \left( x
 a^{-1}  a_{1}\right)  \left(
a_{1}^{-1}  y\right)\\ & =\underbrace{
\left( x  a^{-1}
a_{1}\right)}_{\in B_{a_{1}}}
\underbrace{\left(  a^{-1}  a_{2}
y\right)}_{\in B_{a_{2}}}.
\end{align*}
So \begin{align*} \left(x  a^{-1}
y\right)&= (M_{c}x+\gamma_{1})
(N_{d} y+\gamma_{2})=
P_{e}(M_{c}x+\gamma_{1})+
Q_{f}(N_{d}y+\gamma_{2})+ \beta,
\end{align*} as required. \end{proof}

\begin{lemma}\label{SUPERvecindad} Let $a$ be a dim-generic
element of $G$ and $C$ a cell containing it, as in the previous
lemma. Then there is a $C$-box $B_{a}$ around
$a$, such that for every $ x, y \in
B_{a}$, $x  a^{-1}
y= x - a +y$. \end{lemma}

\begin{proof}
By Lemma  \ref{vecindadcool} there is a $C$-box $B_{a}$
around $a$, matrices $M,N,P,Q \in \mathbb{Q}_{n\times n}$,
tuples $c,d,e,f \in \mathbb{Z}^{n^{2}}$ and
$\gamma_{1},\gamma_{2},\beta \in \mathcal{M}^{n}$
such that for every $x,y \in
B_{a}$,

\[x  a^{-1}  y= P_{e}(M_{c}x+\gamma_{1}+
Q_{f}(N_{d} y+\gamma_{2})+ \beta.\]

In particular $a \in B_{a}$, so we have
\[
x=x  a^{-1}
a=  P_{e}(M_{c}x+\gamma_{1})+
Q_{f}(N_{d} a+\gamma_{2})+ \beta.
\]

Similarly,
\[y=a  a^{-1}
y= P_{e}(M_{c}a+\gamma_{1})+
Q_{f}(N_{d} y+\gamma_{2})+ \beta, \]
so
\begin{align*}
x+y&=  P_{e}(M_{c}x+\gamma_{1})+ Q_{f}(N_{d} a+\gamma_{2})+ \beta+ P_{e}(M_{c}a+\gamma_{1})+ Q_{f}(N_{d} y+\gamma_{2})+ \beta\\
&=\underbrace{ \left(P_{e}(M_{c}x+\gamma_{1})+ Q_{f}(N_{d} y+\gamma_{2})+ \beta \right)}_{x  a^{-1}  y}+\underbrace{\left(P_{e}(M_{c}a+\gamma_{1})+ Q_{f}(N_{d} a+\gamma_{2})+ \beta \right)}_{a a^{-1}  a}.
\end{align*}

Since $ a=a a^{-1}  a$, we have
$x+y= x a^{-1} y+a$, so we can
conclude $x-a+y= x a^{-1} y$.
\end{proof}

\begin{theorem} \label{ab-by-finite} Every group
$(G,\cdot,e_G)$ definable in Presburger arithmetic is definably abelian-by-finite. This is, there is
a definable abelian subgroup $G'$ of $G$ of finite index.\end{theorem}

\begin{proof} By Lemma \ref{SUPERvecindad} we have neighborhood
of the identity such that $\ota$ commutes between any two elements
of the neighborhood. The proof will follow the standard trick of
using a double centralizer.

Let $(G, \cdot, e_{G})$ be a definable group and take a
dim-generic element $a \in G$ in a cell $C$. By Lemma
\ref{SUPERvecindad} there is $C$-box $B_{{a}}$ such that
for every $ {x},{y} \in B_{{a}}$,
${x}  {a}^{-1}  {y}=
{x}-{a}+{y}$.

Given ${x},{y} \in G$, define ${x}
\oplus_{a} {y}= {x}   {a}^{-1}
{y}$, notice that $\oplus_{a}$ is a group operation whose
identity is ${a}$, and $(G,\cdot, e)$ is definably
isomorphic to  $(G,\oplus_{a},{a})$, via the function
$f({x})= {x}  {a}$. It is
therefore enough to show that $(G,\oplus_{a},{a})$ is
abelian-by-finite.

Recall that the centralizer $C(X)$ of a set $X$ is $\{y \in G \mid
\forall (x\in X), \ y\oplus_{a} x=x\oplus_{a} y\}$. Because $\ota$
is commutative in $B_a$, we know that $H=C(C(B_{a}))$ contains
$B_a$ and is therefore a commutative subgroup of $G$. Hence $H$
has the same dimension as $(G,\oplus_{a},{a})$ so it has finite
index by the properties of dimension. \end{proof}

Every abelian-by-finite group is \emph{amenable}, meaning it admits a finitely additive probability
measure on
sets which is invariant under left multiplication. By Theorem \ref{ab-by-finite}
any group $G$ definable in Presburger arithmetic admits such a measure $\mu_G$ on the
algebra of subsets of $G$, so in particular it is a finitely additive probability measure on
definable sets.  So we get the following.

\begin{corollary} \label{groupdefamenable} Every group $G$ definable in Presburger arithmetic is
amenable. In particular, also definably amenable.  \end{corollary}

Given any group $(G,\cdot, e)$ definable in Presburger arithmetic, we will denote by $\mu_G$ the invariant measure given by the previous corollary.

\section{ Generic definable subsets of bounded groups}\label{sec3}

As mentioned in the introduction, we want to show that any bounded
definable group (see Definition \ref{Bounded Set}) is a definable
quotient of $\M^n$ by a lattice. In order to prove this result we
will use results in \cite{HrPi} and \cite{ChSi} to characterize
definable generic subsets of bounded groups. In this section we
prove that if $G$ is a definable bounded group in Presburger, then
a definable subset is generic if and only if it has positive
measure with respect to $\mu_{G}$.

We define a bounded set $X(a)=\phi(\mathcal{M},a)$ as follows,
\begin{definition} [Bounded set]\label{Bounded Set} Let $X(a)=\phi(M,a) \subseteq \mathcal{M}^{n}$ be a definable set,
we say that $X(a)$ is bounded if there is some element $0 < \alpha
\in \mathcal{M}$ such that for any element ${x} \in X(a)$, we have
that for all $i \leq k$, $-\alpha < x_{i}< \alpha$, where $x_{i}$
indicates the $i$-th coordinate of ${x}$.
\end{definition}

\medskip

The characterization of generic sets is inspired by methods  found
in \cite{acotado} and \cite{Dolich}, which contain similar results
for o-minimal theories. In particular, we will begin by proving an
analog of Theorem $2.1$ in \cite{acotado} (see Theorem
\ref{PIFpunto}).

%

\subsection{Definable functions between two non-algebraic types}\label{sub3.1}

In the following statements $p(x_{1},\dots,x_{n})$  is a complete
$n$-type over $\mathcal{M}_{0}$ of dimension $n$. Let $q(y)$
denote a complete non-algebraic $1$-type over $\mathcal{M}_{0}$,
and let $\mathcal{P}, \mathcal{Q}$ be their sets of realizations
in $\M^n$ and $\M$ respectively. By quantifier elimination, $q(y)$
is determined by an $\mathcal{M}_{0}$-cut and by a set of formulas
of the form  $x \equiv_{n} c$ with $0\leq c < n$. We will indicate
as $\mathcal{Q}'$ the set of points that satisfy the same
$\mathcal{M}_{0}$-cut implied by the type $q(y)$, i.e., the
realizations of the type $q(y)\in S_{1}(\mathcal{M}_{0})$
restricted to the language $\mathcal{L}'=\{<\}$.

\begin{lemma} \label{lemasobre} Let $\mathcal{P}$ and $\mathcal{Q}$ be as above. Let ${a} \in \mathcal{P}$, and $\alpha: \mathcal{P} \rightarrow \mathcal{M}$ a relatively $\mathcal{M}_{0}$-definable function, such that $\alpha({a}) \in \mathcal{Q}$. Then $\alpha$ is surjective onto $ \mathcal{Q}$.
\end{lemma}
\begin{proof}
If $\beta \in \mathcal{Q}$ then $\tp(\beta/\mathcal{M}_{0})= \tp(\alpha({a})/\mathcal{M}_{0})$, so there is an automorphism $h \in Aut(\mathcal{M}/\mathcal{M}_{0})$ such that $h(\alpha({a}))=\beta$. Thus, $h({a}) \vDash p$ and $\alpha(h({a}))= \beta$, so $\alpha$ is surjective.
\end{proof}

\begin{lemma}\label{lemaconjugado} Let $\mathcal{P}$ be the set of
realizations of a type $p({x})\in S_n(\mathcal{M}_{0})$ of
dimension $n$ and let $\alpha: \mathcal{P} \rightarrow
\mathcal{M}$ be a relatively $\mathcal{M}_{0}$-definable function.
Take ${a}\in \mathcal{P}$ arbitrary, and suppose that $q(y)=
\tp(\alpha({a})/\mathcal{M}_{0})$ is a non-algebraic 1-type. Let
$\mathcal{Q}'$ be as described at the begining of this subsection.
Then there are elements ${a}^{1}, {a}^{2}\in \mathcal{P}$ such
that $\alpha({a}^{1})< dcl({a}\mathcal{M}_{0}) \cap
\mathcal{Q}'<\alpha({a}^{2})$. \end{lemma}

\begin{proof} Consider the partial type given by \[\Sigma(y)=q(y)\cup
\{\beta({a})<y \ | \text{  $\beta$ is a
$\mathcal{M}_{0}$-definable function and  $\beta({a}) \in
\mathcal{Q}'$}\}.\]

By quantifier elimination, $q(y)$ is completely determined  by its atomic formulas over $\mathcal{M}_{0}$,
and those formulas  determine a linear system of congruences together with formulas of the form $b<x < b'$
where $b,b' \in \mathcal{M}_{0}$.  Let $\Sigma_{0}(y)$ be a finite subset of $\Sigma(y)$. We may assume that
$\Sigma_{0}(y)$ is a conjunction of the form
\[\underbrace{ \bigwedge_{i=1}^{r}(y \equiv_{n_{i}} c_{i})}_{(1)} \wedge  \underbrace{ \bigwedge_{j=1}^{s}b_{j}^{1}< y < b_{j}^{2}}_{(2)} \wedge \underbrace{ \bigwedge_{k=1}^{t}( y > \beta_{k}({a}))}_{(3)},\] where for each $i\leq r$, $0 \leq c_{i} < n_{i}$ are integers; for each $j\leq s$, $b_{j}^{1},b_{j}^{2} \in \mathcal{M}_{0}$
and for all $k\leq t$, $\beta_{k}({x})$ is an $\mathcal{M}_{0}$-definable function.
By the generalized version of the Chinese Reminder Theorem we know that the conjunction (1) can be reduced to a single formula of the form $y \equiv_{N} c$. Additionally, we can reduce $(2)$  to a single formula $d< x <d'$ and let $\beta({a})=\max\{\beta_{i}({a}) \ | \ i\leq n\}$. By cell decomposition, there is a cell $D$ such that $\mathcal{P} \subseteq D$ and $\beta\upharpoonright_{D}$ is an $\mathcal{M}_{0}$-definable linear function.\\

On the other hand, the interval $(\beta({a}), d')$ must be
infinite. Assume by contradiction that it is finite, so there is
$k\in \mathbb{N}$ such that $\beta({a})+k=d'$, but this
contradicts $p({x})\in S_n(\mathcal{M}_{0})$ has dimension $n$
(since $\beta$ is a linear function, we can choose $r \leq n$ the
maximum index such that the coefficient in $\beta$ of $a_{r}$ is
non-zero, isolating $a_{r}$ we can conclude that $a_{r} \in
dcl(a_{1}\dots a_{r-1} d'))$. So let  $j \in \mathbb{N}$ such that
$\beta({a})+j \equiv_{N} c$, obtaining $\mathcal{M} \models
\Sigma_{0}(\beta({a})+j)$.

Thus, by compactness and saturation there is an element
$\alpha_2\models \Sigma(y)$. In particular, we have
$\alpha_2\models q(y)$ so by Lemma \ref{lemasobre} there is
${a}_2\in \mathcal{P}$ such that $\alpha({a}_{2})=\alpha_2$. By
construction of the type $\Sigma(y)$, we have that
$dcl({a}\mathcal{M}_{0})\cap \mathcal{Q}'<\alpha({a}_2)$. We can
show similarly that there is an ${a}_1\in \mathcal{P}$ such that
$\alpha({a}_1)$ realizes the partial type
\[\Sigma_1(y):= q(y)\cup
\{y<\beta({a})\ |\  \text{$\beta$ is a $\mathcal{M}_{0}$-definable
function and $\beta({a})\in\mathcal{Q}'$}\},\] from which we will
obtain $\alpha_1({a})<dcl({a}\mathcal{M}_{0})\cap \mathcal{Q}'$.
\end{proof}

\subsection{Forking and generic formulas in Presburger arithmetic}\label{sub3.2}
\vspace{0.5cm}

We will need the following definitions.

\begin{definition} Let $X$ be a definable subset of $G$. $X$ is said to be
\emph{left generic} if there are $g_{1},\dots,g_{k} \in G$ such
that $\displaystyle{G=\bigcup_{i=1}^{k} g_{i}  X}$. Similarly
we define a right generic set. A formula is called \emph{left (right) generic} if its set of
realizations is left (right) generic in $G$. Given a global type $p$ centered at $G$,
we say that \emph{$p$ is left generic} if every formula in $p$ is left generic in $G$.
If either a set, formula or type is both left and right generic,
we will refer to it simply as \emph{generic}. \end{definition}

\begin{definition}
\begin{enumerate}
\item Define $Def_{G}(A)=\{ X \subseteq G \ | \ \text{$X$ is definable over $A$ }\}$. We will indicate as $Def_G$ the set $Def_{G}(\mathcal{M})$.
\item Let $\mathcal{I}_{G}(A)=\{ X \in Def_{G}(A) \ | \ \text{$X$ is non-generic }\}$. We will be mainly interested in $\mathcal{I}_{G}(\mathcal{M})$ which we will denote as $\mathcal I_{G}$.
\end{enumerate}
\end{definition}

In \cite{acotado} Pillay and Peterzil presented a
characterization of the definable generic subsets of an abelian
and definably compact group in o-minimal theories, using the fact
that for definably compact o-minimal groups, non forking is
essentially equivalent to finite satisfiability over a small model
(cf. Theorem 2.1 in \cite{acotado}). With this idea in mind, we
attempt to find conditions on the definable groups in Presburger
similar to ``definable compactness'' that allow us to recover
analogues to results of  \cite{acotado} in  Presburger arithmetic.
One such condition might be that the group is \emph{bounded}.

The main result of this subsection, Theorem \ref{PIFpunto}, will
be proved by an induction on dimension. The following is the base
case.

\begin{lemma}\label{1celdasdisjuntas}  Let $\phi(x, {a})$ be a
formula defining either  $0$-cell or a bounded $1$-cell. Assume
that $\phi(\mathcal{M}, {a}) \cap \mathcal{M}_{0} = \emptyset$.
Then there is an element ${a}' \in \mathcal{M}^{k}$ such that
\begin{enumerate} \item
$\tp({a}/\mathcal{M}_{0})=\tp({a}'/\mathcal{M}_{0})$, \item
$\phi(\mathcal{M},{a}) \cap \phi(\mathcal{M},{a}')= \emptyset$.
\end{enumerate}

\end{lemma}

\begin{proof} Without loss of generality, we may assume that ${a}$ is a tuple of independent elements over $\mathcal{M}_{0}$. If $\phi(x, {a})$ defines a single point $p \notin \mathcal{M}_{0}$, then the statement is clear. Suppose now that $\phi(x, {a})$ is a bounded $1$-cell, that is, $\phi(x, {a})$ is equivalent to a set defined by a formula of the form $(\alpha({a}) < x < \beta({a})) \wedge x \equiv_{N}c)$ where $\alpha$ and $\beta$ are definable functions and $0 \leq c < N$ are integers. Note that $\alpha({a})$ and $\beta({a})$ are in the same $\mathcal{M}_{0}$-cut. Assume by contradiction that there is $d \in \mathcal{M}_{0}$, such that $\alpha({a})< d < \beta({a})$, since the interval $(\alpha({a}),\beta({a}))$ is infinite, either $(\alpha({a}),d)$ is infinite or $(d,\beta(\alpha))$ is infinite,  so there is an element $m_0\in [d-N,d+N]\cap (\alpha({a}),\beta({a}))$ satisfying $m_0\equiv_N c$. So, $m_0\in \phi(\M,{a})\cap \mathcal{M}_{0}$, but this contradicts our hypothesis.
Let $\mathcal{Q}'$ be the set of points of $\mathcal{M}$ that are
in the same $\mathcal{M}_{0}$-cut as $\alpha({a})$. By Lemma
\ref{lemaconjugado}, there is an element ${a}' \in
\mathcal{M}^{k}$ such that $\tp({a}/\mathcal{M}_{0})=
\tp({a}'/\mathcal{M}_{0})$, and $\alpha({a}') >
dcl(\mathcal{M}_{0}{a}) \cap \mathcal{Q}'$. In particular, since
$\beta({a})\in dcl(\mathcal{M}_{0}{a})\cap \mathcal{Q}'$, we have
$\beta({a}) < \alpha({a}')$. Thus, \[\phi(\M,{a})\cap
\phi(\M,{a}')\subseteq (\alpha({a}),\beta({a}))\cap
(\alpha({a}'),\beta({a}'))=\emptyset.\qedhere \]\end{proof}

For the next result we will need the following, which is proved in
\cite{CoSt} but we use the statement given in Fact 1.1 in
\cite{SiSt}.

\begin{fact}\label{fact1} Let $T$ be a dependent theory,
$\mathcal{M}$ a model and $\mathcal{N}$ a saturated extension of $\mathcal{M}$.

We say that a formula $\phi(x,b;d)$ over $Nb$ \emph{$b$-divides}
over $\mathcal M$ if there is an $\mathcal M$-indiscernible
sequence $\langle d_i\rangle_{i\in \mathbb N}$ inside $N$ with
$d_0=d$ and $\{\phi( x,b ; d )\}$ $k$-inconsistent.

As usual, $b$-forking is then defined as implying a disjunction of
formulas which $b$-divide.

The following are equivalent:

\begin{enumerate}
\item $\phi( x,b ; d )$ does not $b$-divide over $\mathcal{M}$;

\item $\phi( x,b ; d )$ does not $b $-fork over $\mathcal{M}$;

\item If $\langle d_i : i < \omega\rangle$ is a strict Morley
sequence of $tp( d/\mathcal{M} )$ inside $\mathcal{N}$,  then $\{\phi( x,b ; d_i ) :
i < m \}$ is consistent where $m$ is greater than the alternation
number of $\phi( x,y ; z )$.
\end{enumerate}
\end{fact}

We can now prove the analogue of Theorem 2.1 in \cite{acotado}.

\begin{theorem} \label{PIFpunto} Let $X(a):=\phi(\mathcal{M},a) \subseteq
\mathcal{M}^{n}$ be a definable bounded set. The following
statements are equivalent:

\begin{enumerate}

\item $\phi(x,a)$ does not fork over $\mathcal{M}_0$.

\item $\{ \phi(\mathcal{M},y) \ | \ {y} \models
\tp({a}/\mathcal{M}_{0}) \}$ has the finite intersection property.

\item $\phi(\mathcal{M},a) \cap \mathcal{M}_0^{n}\neq \emptyset$.\\
\end{enumerate}
\end{theorem}

\begin{proof}

The direction $(1)\Rightarrow (2)$ follows by the fact that in an
NIP theory any formula $\phi(x,a)$ which does not fork over a
small model $\M_{0}$ is contained in a global $\M_{0}$-invariant
type. The implication  $(3)\Rightarrow (1)$ is clear.

The proof of  $(2)\Rightarrow (3)$ is completely analogous to that
of Lemma 2.2 in \cite{SiSt}, but we include it for completeness.

Let $\phi(x,{a})$ a definable set. We will prove by induction on
$n$ that if $\phi(x,{a})$ does not divide (or fork) over $\mathcal{M}_0$ then
it has a point in $\mathcal{M}_0$.

The case $n=1$ is given by Lemma \ref{1celdasdisjuntas}, assume
that the theorem holds for $n=k$ and assume that $\phi(x,a)\subset
\mathcal{M}^{k+1}$.

So let $x=(x_1, x_2)$ where $x_1$ is a 1-tuple, $x_2$ is a
$k$-tuple, and $\phi(x,a)=\phi(x_1, x_2, a)$. Let $\langle a_i\rangle_{i\in \omega}$ be
a strict Morley sequence of $a$ over $\mathcal{M}_0$. Let $m$ be the
alternating number of $\phi(x_1, x_2; y)$, let
\[
\chi(x_1, x_2, a):= \bigwedge_{i\leq m} \phi(x_1, x_2, a_i),
\]
and let
\[
\theta(x_1, a):= \exists x_2 \chi (x_1, x_2, a).
\]

Let $p(x)$ be a global $\mathcal{M}_{0}$-invariant type extending $\phi(x,a)$.
All $\phi(x,a_i)$ must be in $p(x)$ (by $\mathcal{M}_{0}$-invariance) so
$\phi(x_1, x_2; a)$ is in $p(x)$, and by construction so is
$\theta(x_1, a)$.

This implies of course that $\theta(x_1,a)$ does not fork
over $\mathcal{M}_0$ so that by the 1-variable case we have a point $b_1$ in
$\mathcal{M}_0$ such that $\theta(b_1,a)$ holds, so that $\exists x_2
\chi (b_1, x_2, a)$, and
\[
\bigwedge_{i\leq m} \phi(b_1, x_2, a_i)
\]
is consistent.

By Fact \ref{fact1} this implies that $\phi(b_1, x_2, a)$ does not
fork over $\mathcal{M}_0$ (recall $b_1\in \mathcal{M}_0$) so that by the induction
hypothesis, there is some tuple $b_2\in \mathcal{M}_0$ such that $\phi(b_1,
b_2, a)$ holds, and $(b_1, b_2)\in \phi(x,a)\cap
\mathcal{M}_0^{k+1}$, as required.
\qedhere
\end{proof}

\begin{corollary}\label{acotadoIdeal} Let $G$ be a definable bounded group in Presburger arithmetic,
then $\mathcal{I}_{G}$ is an ideal.
\end{corollary}
\begin{proof}
Let $G$ be a definable bounded group, by Theorem \ref{ab-by-finite}
there is an abelian definable subgroup $H$ of finite index. Following the proof presented
in Section $3$ of \cite{acotado}, replacing definably compact by bounded,
 we conclude that $I_{H}$ is an ideal.
Since $H$ has finite index on $G$,
 $I_{G}$ must be also an ideal in the algebra $Def(G)$.
\end{proof}

Hence, if a finite union of sets is generic, then at least one of them will
be generic. We will use this repeatedly.

\medskip

 The following corollary was already observed in \cite{HrPi} (in
 the  discussion after Problem 5.5 in \cite{HrPi}).

\begin{corollary}\label{fgenericGeneric}
Let $\mathcal{M}_0$ be a small model, and let $X:=\phi(\mathcal{M},a)$ be a definable set  such that $\phi(g^{-1}x,a)$ does not
fork over $\mathcal{M}_0$ for any $g\in G(\mathcal{M})$. Then $X$ is generic.
\end{corollary}

\begin{proof}
Assume that  $\phi(g^{-1}x,a)$ does not fork over $\mathcal{M}_0$ for any $g\in G(\mathcal{M})$. By
Theorem \ref{PIFpunto} every such translate intersects $\mathcal{M}_0$. By
compactness we have that finitely many translates cover
$G(\mathcal{M}_0)$.
\end{proof}

The following will be very useful for us.

\begin{corollary}\label{genericPositive}
Let $G$ be a bounded group definable in Presburger arithmetic, and
let $X$ be a definable subset of $G$. Then $X$ is generic if and
only if $\mu_G(X)>0$.
\end{corollary}
\begin{proof}
If a set is generic, then by invariance and additivity it must
have positive measure. For the converse, by Theorem 1.2 in
\cite{ChSi} we know that a set $X:=\phi(\mathcal{M},a)$  has
positive measure if and only if $\phi(g^{-1}x,a)$ does not fork
for all $g\in G$. By Corollary \ref{fgenericGeneric} this implies
that $X$ is generic, as required.
\end{proof}

\section{Characterizing bounded definable groups in Presburger}\label{sec4}

Let $G$ be any bounded definable group of dimension $n$. Since it
is bounded, we may assume that $G\subseteq \mathcal{M} ^n$, which we will do
throughout this section.

We will need to define generalized parallelograms.

\begin{definition}[Linear Strip]\label{linear strip}
Let $k \in \mathbb{N}$ be such that $k \geq 1$. Let
$q_i:=\frac{s_i}{n_i}$ be rational numbers and let
$f({x}):=\sum_{i=1}^{k} s_{i} \frac{x_{i}}{n_{i}}$, so that for
any $ a$ the $f( a)$ is in the $\mathbb Q$ vector space generated
by the $a_i$'s (so formally not a function in Presburger).
Finally, let $\gamma=\frac{l}{m} d$ with $l,m\in \mathbb Z$ and
$d\in Z$.

By $f({x})<\gamma$ we will understand $cf({x})<k\gamma$ where $c$ is the minimum common multiple of $\{n_i\}_{i\leq k}\cup \{m\}$, so that

\[
cf({x})<c\gamma \Leftrightarrow \sum_{i=1}^{k} s_{i} x_i \frac{c}{n_{i}}<\frac{cl}{m} d
\]
which is a formula in Presburger arithmetic.

\medskip

Given any such function $f$ and any two elements $\gamma_1$ and
$\gamma_2$ such that $\gamma_{2}-\gamma_{1}$ is infinite, we
define the $k$-strip $S_{k}(\gamma_{1},\gamma_{2})$ to be the set
\[
\left\{x= (x_{1},\dots,x_{k}) \in M^{k} \ |   \gamma_{1} \leq f\left({x}\right) \leq \gamma_{2}\right\}.
\]
\end{definition}

Notice that requiring $\gamma_2-\gamma_1$ to be infinite is
equivalent to having equivalently, $m_1l_2d_2-m_2l_1d_1$ to be
infinite where $\gamma_1=\frac{l_1}{m_1} d_1$
$\gamma_2=\frac{l_2}{m_2} d_2$ with $l_1,l_2,m_1,m_2\in \mathbb Z$
and $d_1, d_2\in \M$.

\begin{definition}
We will say that a point $ a$ \emph{lies in the center of a
$K$-strip $S_{k}(\gamma_{1},\gamma_{2})$ } if the distance between
$2f( a)$ and $\gamma_1+\gamma_2$ is finite.
\end{definition}

\begin{remark}\label{BoxinStrip}
Notice that by linearity of the function $f$ defining the strip,
and because we require that $\gamma_2-\gamma_1$ to be infinite,
around any point ${a}$ in the center of the strip we can find an
infinite box around $ a$ contained in the strip.
\end{remark}

\begin{definition} \label{open parallelograms}
An open full $l$-parallelogram is an intersection of $l$ linear strips.

Specifically,
\[
X:=\left\{ \bar x \mid \bigwedge_{i=1}^l \gamma_1^i\leq
f^i\left(\bar x\right)\leq \gamma_2^i\right\}
\]
where $\gamma_1^i, \gamma_2^i$ and $f^i$ satisfy the conditions in
Definition \ref{linear strip} for every $i$.

\bigskip

 An open
$l$-parallelogram is a set $X$ such that
\[
X:=\{{x}=(x_1, \dots, x_l) \mid x_i\equiv_{n_i} c_i \wedge {x}\in  P\}
\] where $P$ is an open full $l$-parallelogram, $0 \leq c_{i}<n_{i}$ are natural numbers and $x_{i}$ is the $i$-th coordinate of $x$.
\end{definition}

The definition allows for degenerate cases: it may be that two
open strips $S_{k}(\gamma_{1}^{i},\gamma_{2}^{i})$ and
$S_{k}(\gamma_{1}^{j},\gamma_{2}^{j})$ are equal.

However, all of the open full $l$-parallelograms we will work with
will be bounded, which will imply that none of this ``degenerate''
cases arise. We will not need to use any of this, which is the
reason we do not make the definition more strict.

\begin{definition} \label{smaller parallelograms}
A full $l$-parallelogram in $M^{n}$ is defined to be the image of
an open full $l$-parallelogram under linear functions, that is, a
set of the form
\[
\left\{\left(f_1\left(a\right), f_2\left( a\right), \dots, f_n\left( a\right)\right) \mid  a \in \mathcal P\right\}
\]
where $f_i\left( a\right)$ is a linear function from $M^k$ for each $i<n$ and
$P$ is an open full $k$-parallelogram.

\medskip

A $k$-parallelogram in $M^n$ is a set  $X$
such that
\[
X:=\{{x}=(x_1, \dots, x_l) \mid x_i\equiv_{n_i} c_i \wedge {x}\in   P\}
\] where $  P$ is an open full $k$-parallelogram, $0 \leq c_{i}<n_{i}$ are natural numbers and $x_{i}$ is the $i$-th coordinate of $x$.
\end{definition}

The definition of full parallelograms
is equivalent to the definition of parallelograms (Definition 3.5) in
\cite{Pantelis}, where the theory is developed in the
context of $Th(\mathbb Q,+,<)$. We will not use their definition, but
we will need some of their more technical results. We will therefore include
their definition and how to ``translate'' between both definitions in Appendix \ref{equivalent definitions}.

In particular, modulo the translation explained in Appendix
\ref{equivalent definitions}, we can repeat the proof of Lemma 3.6
in \cite{Pantelis}, adding the congruences where needed, and get
the following:

\begin{fact}\label{paralelogramo}
Let $W\subseteq \M^n$ be a definable bounded set in Presburger arithmetic.
Then $W$ is a finite union of parallelograms.
\end{fact}

We will later want to combine Fact \ref{paralelogramo} with Lemma
\ref{SUPERvecindad}, and for this we will need to have a
parallelogram with a dim-generic center.

\begin{definition}
We will say that a point $ a$ lies in the center of a an open $k$-parallelogram $  P$ if it lies in the center of all the linear strips defining $  P$.
\end{definition}

\begin{remark}\label{BoxinParallelogram}
Because intersection of boxes containing $ a$ is a box containing $ a$, Remark \ref{BoxinStrip} implies that around any point ${a}$ in the center of an open $k$-parallelogram $  P$ we can find an infinite box around $ a$ contained in the parallelogram.
\end{remark}

We will want to work with parallelograms with specific centers. A
parallelogram   $ P_{ a}$ will always be a parallelogram with
center $ a$. Given such a parallelogram $P_a$ we will need to use
``octants''.

\begin{definition}\label{def: octant}
Let $P_a$ be an open $l$-parallelogram centered at a point $a$
which will be fixed for this definition. Suppose that $P_a$ is
defined by $l$ equations $f^i\left(\bar x\right)\leq \gamma_2^i$.
Let $\eta\in \{-1, 1\}$. Then the \emph{$\eta$-octant of $P_a$} is
the set defined by equations
\[
f^i(a)\leq f^i\left(\bar x\right)\leq \gamma_2^i
\]
if $\eta(i)=1$ and
\[
\gamma_1^i \leq f^i\left(\bar x\right)\leq f^i(a)
\]
if $\eta(i)=-1$ and
\end{definition}

The following holds.

\begin{lemma}\label{LemmaOctants}
Let $P_a$ be an open $l$-parallelogram centered at $a$, and let
$P_a^1$ be an octant of $P_a$ defined by $\eta$. Then:

\begin{enumerate}
\item Any box containing $a$ intersects $P_a^1$ in a set of
dimension $l$.

\item If $x_1,x_2,x_3\in P_a^1$ and $x_1+x_2+x_3-a-a\in P_a^1$,
then $x_1+x_2-2a\in P_a^1$. Here we are using the coordinate
addition in $\mathcal Z^l$.

\end{enumerate}

\end{lemma}

\begin{proof}
The first item follows by our assumption that linear strips are
infinite and definition of center.

For the second one, assume that $x_1+x_2+x_3-a-a\not\in P_a^1$, so
that for some $i$ one of the equations in the definition of
\ref{def: octant} does not hold. Assume that $\eta(i)=1$ so that
\[
f^i(a)\leq f^i\left(x_1+x_2-a\right)\leq \gamma_2^i
\] does not hold.

Since $x_j\in P_a^1$ we know that $f^i(a)\leq f^i\left(x_j\right)$
and $f^i$ is linear, so $f^i\left(x_1+x_2-a\right)\geq f(a)$. So
$f^i\left(x_1+x_2-a\right)> \gamma_2^i$.

However, by hypothesis $f^i\left(x_1+x_2+x_3-a-a\right)\leq
\gamma_2^i$ so
\[
f^i\left(x_1+x_2+x_3-a-a\right)< f^i\left(x_1+x_2-a\right)
\]
which by linearity implies
\[
f^i\left(x_3\right)< f^i\left(a\right),
\]
a contradiction.
\end{proof}

\bigskip

If $W$ is an $A$-definable bounded set, then
the parallelograms that result in Fact \ref{paralelogramo} will also be $A$-definable. However, we can
find a parallelogram with dim-generic center as follows.

\begin{lemma}\label{paralelogramoGenerico}
Let $  P$ be an $A$-definable bounded full open parallelogram.
Then $P$ is a finite union of parallelograms with dim-generic
centers.
\end{lemma}

\begin{proof}
Let $  P$ be a bounded full $A$-definable parallelogram and let $n$ be its dimension,
so that $  P$ is defined by $n$ (non redundant) linear strips, all
definable over $A$. Now, we can cut $  P$ with $n$ hyperplanes
each a parallel to the hyperplanes defining each of the linear
strips (these are bounded by $A$-definable parallel hyperplanes),
and we can take each of these hyperplanes to be a dim-generic
translate of those defining the linear strips.
We end up with $2^n$ parallelograms,
each with dim-generic center, as required.
\end{proof}

\subsection{Characterizing bounded groups definable in Presburger arithmetic}\label{sub3.3}

By cell decomposition, we can decompose $G\times G$ in cells such
that the group multiplication and inversion are linear functions in
each of the cells. Because the non-generic sets from an ideal (Corollary
\ref{acotadoIdeal}) we know that at least one of these cells, say
$W$, is generic in $G\times G$.

\begin{lemma}\label{cajaGenerica}
Let $G$ be a bounded group definable in Presburger arithmetic, and
let $W$ be a generic cell in $G\times G$. There are $U,V$ generics
in $G$ such that $U\times V\subseteq W$.
\end{lemma}

\begin{proof}
It is well known (see for example \cite{AmenableGroups}) that if
$G$ is an amenable group, then so is $G\times G$ equipped with the
following product measure $\mu_{G\times G}$:

For any $W\in G$ define {\small{
\[
\mu_{G\times G}(W)=sup\left\{\sum_{i\leq n}
\left(\mu_G\left(U_i\right)\mu_G\left(V_i\right)\right)\mid \left(
U_i\times V_i\subseteq W \right)\wedge \left(U_i\times V_i\cap
U_j\times V_j=\emptyset\right) , \  n \in \mathbb{N} \right\}.
\]}}

By Corollary \ref{genericPositive} we know that $W$ is generic if
and only if $\mu_{G\times G} (W)$ is positive, and by definition of the product measure
the result follows immediately.
\end{proof}

Recall that $x\otimes_a y:=xa^{-1}y$ and $x\oplus_a y:=x-a+y$. We
will also use the maximum norm. So for $x=(x_{1},\dots,x_{n}) \in
\mathcal{M}^{n}$ we define $|x|:=\max \{|x_i| \ | \ i \leq n\}$
where $|x_i|$ is the usual absolute value. This is clearly
definable in Presburger arithmetic.

\begin{theorem}\label{supersupervecindad}
Let $G\subseteq \M^n$ be a bounded abelian group definable in a model $\M$ of Presburger arithmetic of dimension $n$. Then there is a
generic parallelogram $P_a\subseteq G$ centered in $a\in G$ such that
$G$-addition centered in $a$ coincides with the usual
$\M^n$-addition centered in $a$.
\end{theorem}

\begin{proof}
Let $a:=( a_i)$.

By cell decomposition and Lemma \ref{cajaGenerica} there are
generic cells $U,V\in G$ such that $G$-addition in $U\times V$ is
given by linear functions. By Fact \ref{paralelogramo} we may find
a generic parallelogram $P\subseteq U$. By Lemma
\ref{paralelogramoGenerico} $P$ is a finite union of
parallelograms with dim-generic center, so at least one of these,
say $P_{a}$, is also generic. In a similar way, we can find $b\in
V$ and a generic parallelogram $Q_b\subseteq V$ such that $b$ is a
$\{a\}$-dim-generic center of $Q_b$.

\begin{claim}
There is a box $B_a$ centered around $a$ such that the following hold:

\begin{itemize}
\item $B_a \subseteq U$.

\item $B_a\subseteq a V  b^{-1}$.

\item For any $ {x},{y} \in B_{{a}}$, we have that $ \
{x}  {a}^{-1}  {y}=
{x} - {a} + {y}$.
\end{itemize}
\end{claim}

\begin{proof}
Let $f$ be the linear function on $P_a$ given by $f(x)=x b$ and $g$ be the linear function on $Q_b$ where $g(y)=a y$. The image of parallelograms under
linear functions is a parallelogram, so both $f(P_a)$ and $g(Q_b)$ are generic parallelograms centered in $a b$. By definition of center, there is a box $B_{ab}$ centered in $a b$ contained in both $f(P_a)$ and $g(Q_b)$, so that in particular
\[
B_{ab}\subseteq g(Q_b)\cap f(P_a)\subseteq a V\cap U b.
\] Then $f^{-1}(B_{ab})$ is a parallelogram centered in $a$, contained in $U\cap a V  b^{-1}$. Once again, by definition of center there is a box $B'_a$ centered in $a$ with
$B'_a\subseteq f^{-1}(B_{ab})\subseteq a  V  b^{-1}$.

Let $B''_{{a}}$ be a box around ${a}$ given by Lemma \ref{SUPERvecindad},
so that  for all $ {x},{y} \in B''_{{a}}$, $ \
{x}  {a}^{-1}  {y}=
{x} - {a} + {y}$.  The box $B_a=B''_a\cap B'_a$ is a box around $a$ which will satisfy all the properties of the claim.
\end{proof}

By hypothesis, $x y=A_{{c}}(x)+C_{{d}}(y)+\beta$ where ${c}, {d}$
are constant tuples of standard elements in $\mathbb Z$ and
$\beta$ is a tuple of constants, according to the definition given
in \ref{matriznot}. For notation purposes we will assume that
${c},{d}$ and $\beta$ are all tuples of $0$'s. The general case is
exactly analogous but the notation gets significantly messier. So
we will assume that in $U\times V$ we have $x y=Ax+Cy$ where $A$
and $C$ are matrices in $GL_n(\mathbb Q)$.

\begin{claim} Let $y\in U$ be such that $y  b\in a V$ then
\[a^{-1} y  b=C^{-1}Ay+b-C^{-1}Aa\]
and $a^{-1} y b\in V$. \end{claim}

\begin{proof}
Assume that $y b \in a  V$. Then $a^{-1} y  b \in V$. Since the
group operation over $U \times V$ is linear, we have
$\displaystyle{y b= Ay+Cb}$ and $\displaystyle{a  (a^{-1} y
b)}=Aa+C(a^{-1} y  b)$. Clearly, $a (a^{-1} y b)=y b$, thus
$\displaystyle{a^{-1} y b=C^{-1}Ay+b-C^{-1}Aa}$.
\end{proof}

\begin{lemma}\label{LemmaBoxAndFinite}
For all $y$ in $B_a$ and $x$ in $U$, if $x\oplus_a y\in U$ then
$x\otimes_a y=x\oplus_a y$.
\end{lemma}

\begin{proof} Since $B_a\subseteq U\cap aVb^{-1}$ we can apply the previous claim, so that for all $x$ in $U$ we have
that
\[xa^{-1}yb=Ax+C(C^{-1}Ay+b-C^{-1}Aa)=Ax+Ay-Aa+Cb.\]
It follows that if $x-a+y$ is in $U$ we have
$(x-a+y)b=Ax+Ay-Aa+Cb=xa^{-1}yb$ so $x-a+y=xa^{-1}y$, as required.
\end{proof}

We will now fix an octant $P_a^1$ of $P_a$ defined by $\eta^1\in
\{-1,1\}^n$.

\begin{claim}\label{claim half}
For all $x,y\in P_a^1$ if $x\opa y\in P_a^1$ then $x\ota y=x\opa y$.

So for all $x\in P_a$, if $x\opa x\in P_a$ we have $x\opa x=x\ota x$.
\end{claim}

\begin{proof}
Let
\[X_a:=\left\{y\in P_a^1 \mid \forall x\in P_a^1,\  \left(\left(x\opa
y\right)\in P_a^1\right)\Rightarrow x\ota y=x\opa y\right\}.\]

By lemma \ref{LemmaBoxAndFinite}, $X_a$ contains $B_a\cap P_a^1$, and so it is a subset of $P_a^1$ of
dimension $n$. We will show that it is closed under $\oplus_a$.

To prove this, let $y_1,y_2\in X_a$ and let $x$ be an element in
$P_a^1$ such that $x\opa(y_1\ota y_2)\in P_a^1$. We know that
$y_1\ota y_2=y_1\opa y_2$ by definition of $X_a$ so that $x\opa
(y_1\opa y_2)\in P_a^1$. By Lemma \ref{LemmaOctants} this implies
that $x\opa y_1\in P_a^1$, and since $y_1\in X_a$ we have $x\ota
y_1=x\opa y_1$. We can apply the definition of $X_a$ with $y=y_2$
and $x\opa y_1$ instead of $x$ and we get
\[
x\opa (y_1\opa y_2)=(x\opa y_1)\opa y_2=(x\ota y_1)\opa y_2=x\ota y_1\ota y_2=x\ota (y_1\opa y_2).
\]
It follows that $y_1\opa y_2\in X_a$. So $X_a$ is a subset of the
octant $P_a^1$ containing $B_a\cap P_a^1$ and closed under
$\opa$-vector addition as long as we stay inside the octant
$P_a^1$.

If $X_a\neq P_a^1$, let $y:=( y_i)$ be an element in
$\P_a^1\setminus X_a$ such that
\[
|y-a|=\sum_i |y_i-a_i|=\sum_i \eta(i)(y_i-a_i)
\]
is as small as possible. Because $B_a\cap P_a\in X_a$ we know that
$|y_k-a_k|$ is infinite for some $k$. Let $e_k$ be the vector with
$0$ in all its coordinates for $i\neq k$ and $e_k=\eta(k)$.
Because $P_a^1$ is defined by $\eta$, we have that
$|(y-e_k)-a_i|\in P_a^1$.

But
\[ |(y-e_k)-a|=\sum_i |(y_i-e_k)-a_i|=\sum_i \eta(i)((y_i-e_k)-a_i)
 =\sum_i \eta(i)(y_i-a_i)-1.\]

By minimality, $y-e_k\in X_a$. But $a+e_k\in B_a\cap P_a\subset
X_a$ and $y=(y-e_k)\opa a+e_k)$, so so $y\in X_a$, a
contradiction. This implies that $P_a=X_a$, and the claim follows.
\end{proof}

We will now restrict ourselves to $P_a/2:=\{x\mid x\opa x\opa
v_x\in P_a\}$ for some $v_x\in B_a$.

\bigskip

\begin{claim}
Given any $x\in P_a/2$ and any $z\in P_a$, if $x\oplus z\in P_a$
we have $x\ota z=x\opa z$.
\end{claim}

\begin{proof}
Let \[X:=\left\{x=( x_i)\in P_a \mid \forall z=( z_i)\in P_a, \
\left(\left(x\opa z\in P_a\right)\wedge \left(\bigwedge_i
\left(|x_i-a_i|>|z_i-a_i|\right)\right)\right)\Rightarrow x\ota
z=x\opa z\right\}.\] It is enough to show that $P_a/2\subseteq X$.

Suppose that $P_a/2\not\subseteq X$, and let $x\in P_a/2$ and
$z\in P_a$ with $|z_i-a_i<|x_i-a_i|$, $x\opa z\in P_a$ and $x\ota
z\neq x\opa z$. We may take such $x$ such that $|x-a|$ is the
smallest possible.

Note that since we have an infinite cube inside $B_a$ such cube
must be a subset of $X$ by Lemma \ref{LemmaBoxAndFinite} $|x-a|$
is therefore infinite.

Now, let $z'=( z_i'), x'=( x_i')$ in $P_a$ and $v_z, v_x$ in $B_a$
be such that $z=z'\opa z'\opa v_z$ and $x=x'\opa x'\opa v_x$.

We can choose the above (adding elements in $B_a$ to $v_x$ and
$v_z$ if needed) so that $|x'_i-a_i|>|z'_i-a_i|$.

By definition of $P_a$ we know that because $x\oplus z\in P_a$
then $x'\opa z'\in P_a$: If $f$ is $\mathbb Q$-linear, $|v-a|$ is
finite, and $\gamma-f(a)$ infinite, then
\[f(w'\opa
w'\opa v)=f(w')+f(w')-f(a)+f(v-a))\leq \gamma ;\] But $f(v-a)$
would be finite so
\[2f(w')<\gamma+f(a)-f(v-a)<2\gamma\]
and $f(w')<\gamma$.

By minimality of $x$ we have $x'\in X$ and $x'\opa z'=x'\ota z'$.
Recall that Claim \ref{claim half} if $w\opa w\in P_a$ then
 $w\opa w=w\ota w$ for any $w$. So $(x'\opa z')\ota (x'\opa
z')=(x'\opa z')\opa (x'\opa z')$, $x'\opa x'=x'\ota x'$ and
$z'\opa z'=z'\ota z'$. Finally, combining this with the last
statement of Lemma \ref{LemmaBoxAndFinite} we have $x=x'\ota
x'\ota v_x$ and $z=z'\ota z'\ota v_z$.

Using the above (and abelianity of $\opa $ and $\ota$) we have:

\begin{align*}
x\opa z=&
x'\opa x'\opa v_x\opa z'\opa z'\opa v_z\\
=& [(x'\opa z')\opa (x'\opa z')]\opa v_x\opa v_z\\
=& [(x'\opa z')\ota (x'\opa z')]\opa v_x\opa v_z\\
=& [(x'\ota z')\ota (x'\ota z')]\opa (v_x\opa v_z)\\
=& [(x'\ota z')\ota (x'\ota z')]\ota (v_x\opa v_z) \text{ By Lemma \ref{LemmaBoxAndFinite},  $v_x\opa v_z\in B_a$}\\
=& [(x'\ota z')\ota (x'\ota z')]\ota (v_x\ota v_z) \text{ again, $v_x, v_z\in B_a$}\\
=& x'\ota x'\ota v_x\ota z'\ota z'\ota v_z\\
=& x\ota z,
\end{align*}
a contradiction.
\end{proof}

It follows both that $x\opa y=x\ota y$ for all elements in
$P_a/2$ and that
$P_a/2$ is generic in $G$ (we can get finitely close to any point of $P_a$
with finitely many $\opa$-translates of $P_a/2$ by elements of $P_a/2$). So the Theorem \ref{supersupervecindad} follows.
\end{proof}

\bigskip

By Theorem \ref{ab-by-finite} any group $(G, \cdot,e)$ definable in
Presburger arithmetic is abelian-by-finite. We will
conclude this paper by characterizing all bounded abelian groups $(G,
\oplus,e)$ definable in Presburger arithmetic.

\begin{definition}\label{locallattice} Let $B \subseteq Z^{k}$ be a box around $0$. Let $\mathcal{B}$ be the $\bigvee$-definable group
$ \displaystyle{\bigcup_{n \in \mathbb{N}} nB}$
where $nB:=\{ b_1+b_2+\dots+b_n \mid b_i\in B\}$
with the natural additive structure $(\mathcal{B},+,0)$.

We define a \emph{local $B$-lattice $\Lambda$ over $\mathcal{B}$} to be a subgroup $\Lambda$ such that $\forall \lambda \in \Lambda$ we have
$(\lambda+B) \cap \Lambda= \{ \lambda\}$.

\end{definition}

\begin{theorem}\label{THEOREM}
Let $(G, \cdot,e)$ be any bounded group definable in
Presburger arithmetic. Then there is an abelian
finite subgroup $G_0$ of $G$ of finite index, a finite integer $k$,
an infinite open box $B\subseteq \M^k$ centered at $0$
and a local $B$-lattice $\Lambda$ in $\M^k$ such that $G_0$ is
definably isomorphic to $\mathcal{B}/\Lambda$.
\end{theorem}

\begin{proof}
By definition, any cell is definably isomorphic to an
$n$-dimensional cell $C$ in $\M^n$. So we may assume that $G\subseteq
\M^k$ with $k=dim(G)$, because $G$ is a bounded group.

By Theorem \ref{ab-by-finite} there is a definable abelian
subgroup $G_{ab}$ of $G$ of finite index. We may of course assume
that $G=G_{ab}$ so we will assume $G$ is abelian. Because of this
and to clarify the notation in the rest of the proof, we will switch to additive notation and have  $(G,
\cdot,e)=(G, \oplus,e)$.

\bigskip

By Theorem \ref{supersupervecindad} we can find a definable
parallelogram $P_a\subseteq G$ centered in $a$ which is generic in
$G$ and such that $xa^{-1}y=x-a+y$ for all $x,y\in P_a$. We may of
course definably shift the operations and assume that $a$ is both
the identity in $G$ and the origin in $\M^k$.

Furthermore, any parallelogram centered in the origin of dimension
$k$ is definably isomorphic to a $k$-box in $\M^k$ centered in the
origin via a linear function, and such an isomorphism will preserve
addition.

So we have a definable local isomorphism $f$ from a box $B\subseteq
\M^k$ centered in the origin into a generic subset of $G$, with the
following properties:

\begin{itemize}
\item For any $n \in \mathbb{N}$ and $x_{1},\dots,x_{n} \in B$, if
$x_{1}+\dots+x_{n} \in B$ then
\[f(x_{1}+\dots+x_{n})= \bigoplus_{i=1}^{n}f(x_{i}).\]
In particular, for any $x \in B$ and $z \in \mathbb{Z}^{k} \cap B$,
if $x+z \in B$ then $f(x+z)= f(x)\oplus f(z)$.

\item  For every $x,y \in B$, we have $f(x) \oplus f(y)=f(y)
\oplus f(x)$. \end{itemize}

Now, we define $f_n: nB\rightarrow
G$ by
\[
f_n(b_1+b_2+\dots +b_n)= f(b_1)\oplus f(b_2)\oplus \dots \oplus
f(b_n).
\]

\begin{claim}\label{well defined}
Each function $f_n$ is well defined.
\end{claim}

\begin{proof}
Fix a natural number $n \in \mathbb{N}$.
Assume that $B= I_{1}\times \dots \times I_{k}$, where each
 $I_{i}=[ -\alpha_{i}, \alpha_{i}] \wedge x\equiv_{N_{i}}0$ and define $R_{i}= nN_{i}$.

We will need to work with the coordinates of the elements in
$\mathcal Z^k$. For this claim we will use $v(i)$ to denote the
$i$'th coordinate of $v$ for any $v\in \mathcal Z^k$.

 Take an element ${v} \in B$, for each $i \leq k$ we can find $w(i), z(i) \in I_{i}$
  such that $v(i)=w(i)+z(i)$, $z(i) \in \mathbb{Z}$ and $w(i) \equiv_{R_{i}} 0$.
  Thus, we can find vectors ${w} \in n B$, ${z} \in \mathbb{Z}^{k} \subseteq B$
  such that ${v}= {w}+{z}$ and each coordinate $w(i) \equiv_{R_{i}} 0$.

 Assume now that $x_{1}+\dots+x_{n}= y_{1}+\dots+y_{n}$, where  $x_{i},y_{i} \in B$ for each $i \leq n$.
 Decomposing vectors ${x}_{i}={w}_{i}+{z}_{i}$
 and ${y}_{i}={t}_{i}+{z'}_{i}$ with ${w_i} \in n B$, $z'_i, {z_i} \in \mathbb{Z}^{k} \subseteq B$
 and ${t}_{i}(j), w_i(j) \equiv_{R_{j}} 0$ for all $0\leq j\leq k$.

Replacing, we get
\[
\underbrace{(w_{1}+\dots+w_{n})}_{W}+\underbrace{
(z_{1}+\dots+z_{n}-z'_{1}-\dots- z'_{n})}_{ Z }=
 \underbrace{(t_{1}+\dots+t_{n})}_{T}.\]

So $Z \in \mathbb{Z}^{k}$, for each $i \leq k$, $Z(j)=T(j)-(j)
\equiv_{R_{j}}0$, and $\frac{W}{n}+\frac{Z}{n}= \frac{T}{n} \in B$. \\

\noindent \textit{Subclaim:}$ \displaystyle{\bigoplus_{i=1}^{n}
f(w_{i}) \oplus f(Z)= \bigoplus_{i=1}^{n} f(t_{i})  }$.
\begin{proof} For each $j \leq n$, $\displaystyle{w_{j}= \underbrace{
\frac{ w_{j}}{n}+\dots+ \frac{ w_{j}}{n}}_{n-\text{times}}}= n
\frac{w_{j}}{n}$ and $\displaystyle{\frac{w_{j}}{n} \in B}$.
Likewise, $\displaystyle{t_{j}=n \frac{t_{j}}{n}}$  and
$\displaystyle{Z= n \frac{Z}{n}}$, where
$\displaystyle{\frac{t_{j}}{n}, \frac{Z}{n} \in B}$. Thus,
\begin{align*}
\bigoplus_{i=1}^{n} f(w_{i}) \oplus f(Z) &=
\bigoplus_{i=1}^{n}f\left(n \frac{w_{i}}{n}\right)
\oplus f\left(n \frac{Z}{n}\right)\\
&= \bigoplus_{i=1}^{n} \underbrace{f \left(\frac{w_{i}}{n} \right)
\oplus \dots \oplus f \left(\frac{w_{i}}{n} \right) }_{
n-\text{times}} \oplus
 \underbrace{f \left(\frac{Z}{n}\right) \oplus \dots \oplus f \left(\frac{Z}{n}\right)}_{n- \text{times}}\\
 &= \underbrace{ \left(\bigoplus_{i=1}^{n} f\left( \frac{w_{i}}{n}\right) \oplus f\left( \frac{Z}{n}\right)\right) \oplus \dots
 \oplus \left(\bigoplus_{i=1}^{n} f\left(
 \frac{w_{i}}{n}\right) \oplus f\left( \frac{Z}{n}\right)\right)}_{n- \text{times}}
 \\
 &=  \underbrace{ f\left( \frac{w_{1}}{n}+ \dots +\frac{w_{n}}{n} + \frac{Z}{n}\right)\oplus \dots
 \oplus f\left( \frac{w_{1}}{n}+ \dots +\frac{w_{n}}{n} + \frac{Z}{n}\right)}_{n- \text{times}} \\
&=  \underbrace{  f\left( \sum_{i=1}^{n} \frac{t_{i}}{n} \right)
\oplus \dots
 \oplus  f\left( \sum_{i=1}^{n} \frac{t_{i}}{n} \right)}_{n- \text{times}} \\
 &=  \underbrace{  \left(\bigoplus_{i=1}^{n} f \left(
 \frac{t_{i}}{n}\right)\right)
\oplus \dots
 \oplus   \left(\bigoplus_{i=1}^{n} f \left( \frac{t_{i}}{n}\right)\right)}_{n- \text{times}} \\
&= \bigoplus_{i=1}^{n} \underbrace{f \left( \frac{t_{i}}{n}\right)  \oplus \dots \oplus f \left( \frac{t_{i}}{n}\right)}_{n- \text{times}} \\
 &=  \bigoplus_{i=1}^{n}f\left(n \frac{t_{i}}{n}\right) \\
 &= \bigoplus_{i=1}^{n}f(t_{i}).
\end{align*}
\end{proof}

Since $\displaystyle{Z= \sum_{i=1}^{n} z_{i}
-\sum_{i=1}^{n}z{'}_{n}}$, we have $ f(Z)=\displaystyle{
\bigoplus_{i=1}^{n} f(z_{i}) \ominus \bigoplus_{i=1}^{n}
f(z'_{i})}$. Hence
\begin{center}
$\displaystyle{\bigoplus_{i=1}^{n} f(w_{i}) \oplus
\bigoplus_{i=1}^{n}f(z_{i})= \bigoplus_{i=1}^{n} f(t_{i}) \oplus
\bigoplus_{i=1}^{n} f(z_{i}')}.$
\end{center}
Reorganizing this equation we obtain
$\displaystyle{\bigoplus_{i=1}^{n} f(w_{i}) \oplus f(z_{i})=
\bigoplus_{i=1}^{n} f(t_{i}) \oplus f(z_{i}')}$, and thus
$\displaystyle{\bigoplus_{i=1}^{n} f(w_{i} +z_{i})=
\bigoplus_{i=1}^{n} f(t_{i} +z_{i}')}$, concluding that $
\displaystyle{\bigoplus_{i=1}^{n} f(x_{i})=
\bigoplus_{i=1}^{n}f(y_{i})}$, as required.
\end{proof}

\noindent \begin{claim}
For some $n$ we have $f_n(nB)=f_l(lB)$ for all $l>n$.
\end{claim}

\begin{proof}
By genericity, $\displaystyle{G=\bigcup_{g\in X} g \oplus f(B)}$ for a finite set $X$.\\
Define $X'=\{g\in X \mid \exists m\in \mathbb N\ g\in f_m(mB)\}$.
Since $f_{m}(mB) \subseteq f_{l}(lB)$ whenever $m \leq l$ and $X'$
is finite, there is some minimal $m \in \mathbb{N}$ such that
 $X' \subseteq f_m(mB)$. It is clear from the construction that $m+1$ has the required properties.
\end{proof}
\noindent Let $\Lambda_n=\{b\in nB\mid f_n(b)=0_G\}$ and
$\displaystyle{\Lambda=\bigcup_{n \in \mathbb{N}} \Lambda_n}$, and
let $ \displaystyle{G_0=\bigcup_{n \in \mathbb{N}}f_n(B_n)}$. By
Claim \ref{well defined} $G_0$ is a definable subgroup of $G$.
Since $G_{0}$ contains $f_1(B)$
it is generic, and so it must have finite index.\\
Since $\lambda\cap B=\{ 0\}$, $\Lambda$ is a local $B$-lattice and $\mathcal B/\Lambda$ is isomorphic to $G_0$.
\end{proof}

\appendix

\section{Translation between parallelograms}\label{equivalent definitions}
\subsection{Equivalent definitions}
Many of the results in this paper were achieved by Eleftheriou and Starchenko
in the context of ordered divisible abelian groups. We made strong use of their
results in Section \ref{sub3.3}, in particular with Fact \ref{paralelogramo}, claiming
that their proof that any cell (in the context of $Th(\mathbb Q,+,<)$)
is a finite union of parallelograms can be applyed in the Presburger case.

This of course can only work if their definition of parallelogram has a direct translation
to the context of Presburger, which coincides with Definitions \ref{smaller parallelograms} and
\ref{open parallelograms}, or
at the very least that these definitions include the analogues of parallelograms defined in
\cite{Pantelis}.

We will now describe the definition of parallelograms from \cite{Pantelis},
how we can apply this definition to Presburger arithmetic, and why these analogues
are $l$-parallelograms.
In this appendix, we will use letters $a,b$ to denote tuples, $c,d$ for elements
in the ground model of $Th(\mathbb{Q},+,<)$, and greek letters for constants in $\mathbb Q$.

\begin{definition}\label{linear case}
Let $M$ be a model of $Th(\mathbb Q, +, <)$. Let $b_1,\dots, b_j$ be elements such that
$b_i=\langle \beta_{i_1} d_i, \beta_{i_2} d_i, \dots, \beta_{i_n} d_i\rangle$ where
$\beta_{i_j}\in \mathbb Q$
and $d_i\in M$. Let
$a=\langle \alpha c_1 , \alpha c_2, \dots, \alpha c_n\rangle=\alpha\langle c_1, \dots, c_n\rangle$ with
$\alpha\in \mathbb Q$ and $\langle c_1, \dots, c_n\rangle\in M^n$.

Then the $n$-parallelogram $P_a(b_1, \dots , b_j)$ anchored at $a$ and determined by $b_1, \dots , b_j$
is the set of
points in $M^n$ that can be written as
\[
a  +  \sum_{i=1}^j \langle \beta_{i_1} t_i, \beta_{i_2} t_i, \dots, \beta_{i_n} t_i\rangle
\]
for some $t_1, \dots, t_n$ in $M$ with $t_i<d_i$.
\end{definition}

The following is Lemma 3.6 in \cite{Pantelis}.

\begin{fact}\label{Pantelis}
The closure of every bounded $n$-dimensional linear cell $Y\subseteq M^n$ is a finite union
of $n$-parallelograms (as in Definition \ref{linear case}).
\end{fact}

Cells in Presburger have a very similar definition as cells in $Th(\mathbb{Q},<,+)$, except for the
congruences. But congruences can be dealt with at any stage (they are the difference between
$l$-parallelograms and full $l$-parallelograms). We will therefore show what the analogue
of Definition \ref{linear case} would be when not taking congruences into account, and we will
prove we get full parallelograms.

\begin{definition}\label{PA} Let $Z$ a model of Presburger arithmetic.
Let $b_1,\dots, b_j$ be elements in $Z^n$ such that
$b_i=\langle \beta_{i_1} d_i, \beta_{i_2} d_i, \dots, \beta_{i_n} d_i\rangle$ where
$\beta_{i_j}\in \mathbb Q$
and $d_i\in Z$. Let
$a=\langle \alpha c_1 , \alpha c_2, \dots, \alpha c_n\rangle=\alpha\langle c_1, \dots, c_n\rangle$ with
$\alpha\in \mathbb Q$ and $\langle c_1, \dots, c_n\rangle\in Z^n$.

Given $\langle a, b_1, \dots, b_l\rangle$ as above, let the full
parallelogram $P_a(b_1, \dots , b_j)$ generated by $\langle a,
b_1, \dots, b_l\rangle$ be the subset of points in $Z^n$ that can
be written as
\[
a  +  \sum_i \langle \beta_{i_1} t_i, \beta_{i_2} t_i, \dots, \beta_{i_n} t_i\rangle
\]
for some $t_1, \dots, t_n$ in $Z$ with $t_i<d_i$.

We may always assume without loss of generality that
$\{\langle \beta_{i_1}, \beta_{i_2}, \dots, \beta_{i_n} \rangle\}_{1\leq i\leq j}$
are $\mathbb Q$-linearly independent.
\end{definition}

We will now show that the sets $P_a(b_1, \dots , b_l)$ as defined above are full $l$-parallelograms
as in Definition \ref{open parallelograms} and \ref{smaller parallelograms}. We will start
with the open case.

\begin{theorem}\label{equivalent open}
Let $a, b_1, \dots, b_n$ be as in Definition \ref{PA} with $j=n$.
Then $P_a(b_1, \dots , b_n)$ is an open full $n$-parallelogram.
\end{theorem}

\begin{proof}
Fix any $j\leq n$. Working within $\mathbb Q$, let $\mathcal P$ be the subspace of $\mathbb Q^n$
generated by the $n-1$-vectors $\langle \beta_{i_1}, \beta_{i_2}, \dots, \beta_{i_n} \rangle$
with $i\neq j$, and assume such a subspace is defined by the equation $H_j( x)=0$. This is an
equation with coefficients in $\mathbb Q$, and it follows that any point in $Z^n$ of the form
\[
\sum_{i\neq j} \langle \beta_{i_1}, \beta_{i_2}, \dots, \beta_{i_n} \rangle t_i
 \]
will satisfy $H_j(\bar x)=0$.

So any point of the form
\[
a+\sum_{i\neq j} \langle \beta_{i_1}, \beta_{i_2}, \dots, \beta_{i_n} \rangle t_i
\]
will satisfy $H_j(\bar x)=H_j(a)$ and
any point of the form
\[
a+b_j +\sum_{i\neq j} \langle \beta_{i_1}, \beta_{i_2}, \dots, \beta_{i_n} \rangle t_i
\]
will satisfy $H_j(\bar x)=H_j(a+b_j)$. Assuming that $H_j(a)< H_j(a+b_j)$ (otherwise we reverse
the order), we
have that the set

\[
\{\bar x\in Z^n \mid H_j(a)\leq H_j(\bar x)\leq H_j(a+b_j) \}
\]
 is precisely the set
\[
a  +  \sum_i \langle \beta_{i_1} t_i, \beta_{i_2} t_i, \dots ,\beta_{i_n} t_i\rangle
\]
where $t_1, \dots, t_n$ vary in $Z$ for $i<j$ and $0\leq t_j\leq d_j$.

So
\[
P_a(b_1, \dots, b_n)=\bigcap_{j=1}^n \{\bar x\in Z^n \mid H_j(a)\leq H_j(\bar x)\leq H_j(a+b_j) \}
 \]
as required.
\end{proof}

\begin{theorem}
Let $a, b_1, \dots , b_l$ be elements in $Z^n$ satisfying the conditions described
in Definition \ref{PA}.

Then $P_a(b_1, \dots , b_l)$ is an open full $l$-parallelogram in $Z^n$ as in Definition
\ref{smaller parallelograms}.
\end{theorem}

\begin{proof}
In this case, let $H(\bar x)=0$ be the $l$-dimensional subspace of $\mathbb Q^n$ generated
by the $\mathbb Q$-vectors
\[
\{\langle \beta_{i_1}, \beta_{i_2}, \dots, \beta_{i_n} \rangle\}_{1\leq i\leq l}.
 \]

By linear algebra, there are $l$ free coordinates and $n-l$ dependent ones in $H(\bar x)=0$. Without
loss of generality, we will assume that the independent ones are the first $l$, so that
\[
\left\{\bar x \in \mathbb Q \mid H\left(\bar x\right)=0\right\}=
\left\{(x_1,\dots, x_j, f_{j+1}\left(x_1, \dots x_j\right),f_{j+2}\left(x_1, \dots x_j\right),
\dots, f_{n}\left(x_1, \dots x_j\right) \text{ with $x_1,\dots, x_j\in \mathbb Q^j$}\right\}
 \]
where $f_{j+1}\left(x_1, \dots x_j\right),f_{j+2}\left(x_1, \dots x_j\right),
\dots, f_{n}\left(x_1, \dots x_j\right)$ are $\mathbb Q$-linear functions. Let
\[\bar f(x_1, \dots , x_n):=\left(x_1,\dots, x_j, f_{j+1}\left(x_1, \dots x_j\right),f_{j+2}\left(x_1, \dots x_j\right),
\dots, f_{n}\left(x_1, \dots x_j\right)\right).\]

Now, the projections $b_1', b_j'$ of
$b_1, \dots b_j$ to
the first $j$ coordinates determine an open full parallelogram
$P_0(b_1',\dots, b_j')$. Then
\[
\bar f\left(P_0\left(b_1',\dots, b_j'\right)\right)=P_0(b_1, \dots, b_n)
\]
and
\[
a+\bar f\left(P_0\left(b_1',\dots, b_j'\right)\right)=P_a (b_1, \dots, b_n).
\]

By Theorem \ref{equivalent open} and the definition of $l$-parallelograms in $Z^n$, the
theorem follows.
\end{proof}




\section*{Acknowledgments}
We would like to thank the anonymous referee for pointing out mistakes in earlier versions of the proof.
We would also like to thank Pantelis Eleftheriou for some very insightful comments which allowed us to improve this paper.


\bibliographystyle{abbrv}
\bibliography{GroupsinPresburger}

\begin{bibdiv}
\begin{biblist}

\bib{quasiomin}{article}{
      author={Belegradek, Oleg},
      author={Peterzil, Ya'acov},
      author={Wagner, Frank},
       title={Quasi-o-minimal structures},
        date={2000},
        ISSN={0022-4812},
     journal={J. Symbolic Logic},
      volume={65},
      number={3},
       pages={1115\ndash 1132},
         url={https://doi.org/10.2307/2586690},
      review={\MR{1791366}},
}

\bib{ChSi}{article}{
      author={{Chernikov}, A.},
      author={{Simon}, P.},
       title={{Definably amenable NIP groups}},
        date={2015-02},
     journal={ArXiv e-prints},
      eprint={1502.04365},
}

\bib{cell}{article}{
      author={Cluckers, Raf},
       title={Presburger sets and {$p$}-minimal fields},
        date={2003},
        ISSN={0022-4812},
     journal={J. Symbolic Logic},
      volume={68},
      number={1},
       pages={153\ndash 162},
         url={https://doi.org/10.2178/jsl/1045861509},
      review={\MR{1959315}},
}

\bib{CoSt}{article}{
      author={Cotter, Sarah},
      author={Starchenko, Sergei},
       title={Forking in {VC}-minimal theories},
        date={2012},
        ISSN={0022-4812},
     journal={J. Symbolic Logic},
      volume={77},
      number={4},
       pages={1257\ndash 1271},
         url={https://doi.org/10.2178/jsl.7704110},
      review={\MR{3051624}},
}

\bib{Dolich}{article}{
      author={Dolich, Alfred},
       title={Forking and independence in o-minimal theories},
        date={2004},
        ISSN={0022-4812},
     journal={J. Symbolic Logic},
      volume={69},
      number={1},
       pages={215\ndash 240},
         url={https://doi.org/10.2178/jsl/1080938838},
      review={\MR{2039358}},
}

\bib{EaOn}{article}{
      author={Ealy, Clifton},
      author={Onshuus, Alf},
       title={Characterizing rosy theories},
        date={2007},
        ISSN={0022-4812},
     journal={J. Symbolic Logic},
      volume={72},
      number={3},
       pages={919\ndash 940},
         url={https://doi.org/10.2178/jsl/1191333848},
      review={\MR{2354907}},
}

\bib{Pantelis}{article}{
      author={Eleftheriou, Pantelis~E.},
      author={Starchenko, Sergei},
       title={Groups definable in ordered vector spaces over ordered division
  rings},
        date={2007},
        ISSN={0022-4812},
     journal={J. Symbolic Logic},
      volume={72},
      number={4},
       pages={1108\ndash 1140},
         url={https://doi.org/10.2178/jsl/1203350776},
      review={\MR{2371195}},
}

\bib{GroupsMeasuresNIP}{article}{
      author={Hrushovski, Ehud},
      author={Peterzil, Ya'acov},
      author={Pillay, Anand},
       title={Groups, measures, and the {NIP}},
        date={2008},
        ISSN={0894-0347},
     journal={J. Amer. Math. Soc.},
      volume={21},
      number={2},
       pages={563\ndash 596},
         url={https://doi.org/10.1090/S0894-0347-07-00558-9},
      review={\MR{2373360}},
}

\bib{HrPi}{article}{
      author={Hrushovski, Ehud},
      author={Pillay, Anand},
       title={On {NIP} and invariant measures},
        date={2011},
        ISSN={1435-9855},
     journal={J. Eur. Math. Soc. (JEMS)},
      volume={13},
      number={4},
       pages={1005\ndash 1061},
         url={https://doi.org/10.4171/JEMS/274},
      review={\MR{2800483}},
}

\bib{Marker}{book}{
      author={Marker, David},
       title={Model theory},
      series={Graduate Texts in Mathematics},
   publisher={Springer-Verlag, New York},
        date={2002},
      volume={217},
        ISBN={0-387-98760-6},
        note={An introduction},
      review={\MR{1924282}},
}

\bib{acotado}{article}{
      author={Peterzil, Ya'acov},
      author={Pillay, Anand},
       title={Generic sets in definably compact groups},
        date={2007},
        ISSN={0016-2736},
     journal={Fund. Math.},
      volume={193},
      number={2},
       pages={153\ndash 170},
         url={https://doi.org/10.4064/fm193-2-4},
      review={\MR{2282713}},
}

\bib{SiSt}{article}{
      author={Simon, Pierre},
      author={Starchenko, Sergei},
       title={On forking and definability of types in some {DP}-minimal
  theories},
        date={2014},
        ISSN={0022-4812},
     journal={J. Symb. Log.},
      volume={79},
      number={4},
       pages={1020\ndash 1024},
         url={https://doi.org/10.1017/jsl.2014.45},
      review={\MR{3343527}},
}

\bib{AmenableGroups}{book}{
      author={Wagon, Stan},
       title={The {B}anach-{T}arski paradox},
   publisher={Cambridge University Press, Cambridge},
        date={1993},
        ISBN={0-521-45704-1},
        note={With a foreword by Jan Mycielski, Corrected reprint of the 1985
  original},
      review={\MR{1251963}},
}

\end{biblist}
\end{bibdiv}

\end{document}